\documentclass[11pt,a4paper]{article}
\usepackage{fancyhdr}
\usepackage{amsmath}
\usepackage{amsthm}
\usepackage{amssymb}
\usepackage{wasysym}
\usepackage{paralist}
\usepackage{makeidx}
\usepackage[all]{xy}
\usepackage{mathdots}
\usepackage{yhmath}
\usepackage{mathtools}

\usepackage{young}
\usepackage[vcentermath]{youngtab}

\usepackage{graphicx}
\usepackage{epstopdf}

\input xy
\newcommand{\nc}{\newcommand}
 \nc{\cl}{\centerline}

 \nc{\SL}{{\rm SL}}
 
 \nc{\Stind}{{\rm Stind}}
 
 \nc{\len}{{\rm len}}
 
 \nc{\Comod}{{\rm Comod}}
 
 \nc{\Diag}{{\rm Diag}}
 
 \nc{\comod}{{\rm comod}}
 
  \nc{\Am}{{\rm Am}}
    \nc{\Cosat}{{\rm Cosat}}
 \nc{\hatQ}{{\hat Q}}
 \nc{\sgn}{{\rm sgn}}
 \nc{\ad}{{\rm ad}}
 
  \nc{\idx}{{\rm index}}
 
 \nc{\Mat}{{\rm Mat}}
 
 \nc{\trace}{{\rm trace}}
 
 \nc{\Tor}{{\rm Tor}}
 
  \nc{\Ann}{{\rm Ann}}
  
  \nc{\Grot}{{\rm Grot}}
  
  \nc{\Ker}{{\rm Ker}}
  
  \renewcommand{\Im}{{\rm Im}}

  \nc{\calA}{{\mathcal A}}
  
  \nc{\A}{{\mathcal A}}
  
    \nc{\B}{{\mathcal B}}
    
      \nc{\C}{{\mathcal C}}
      
           \nc{\DD}{{\mathcal D}}
 
 \nc{\Loewy}{{\rm Loewy}}
 
 \nc{\Supp}{{\rm Supp}}
 
 \nc{\env}{{\rm env}}

 \nc{\wt}{{\rm wt}}
 
 \nc{\poll}{{\rm \l}}
 
 \nc{\efp}{{\Bbb F}_p}

\nc{\baru}{{\overline u}}

\nc{\baralpha}{{\overline \alpha}}

\nc{\bargamma}{{\overline \gamma}}

\nc{\barA}{{\bar A}}

\nc{\barq}{{\overline q}}

\nc{\barQ}{{\overline Q}}

\nc{\bara}{{\overline a}}

\nc{\barx}{{\overline x}}

\nc{\bary}{{\overline y}}

\nc{\barf}{{\overline f}}

 \nc{\Orb}{{\rm Orb}}
 
 \nc{\Cent}{{\rm Cent}}
 
 \nc{\gothg}{{\frak g}}
 
 \renewcommand\implies{$\Rightarrow$}

 \nc{\hatsigma}{{\hat \sigma}}

 \nc{\hatpi}{{\hat \pi}}
 
  \nc{\hatzeta}{{\hat \zeta}}

 \nc{\Ocal}{{\mathcal O}}

 \nc{\M}{\mathfrak{m}}
 
 \nc{\seee}{\mathbb C}
 
 \nc{\varleq}{\preccurlyeq}
 
 \newcommand{\id}{{\rm id}}

 \nc{\hatlambda}{{\hat\lambda}}
 
 \nc{\hatphi}{{\hat \phi}}
 
 \nc{\daggerlambda}{{\lambda^\dagger}}

    \nc{\barr}{{\bar r}}
  \nc{\bart}{{\bar t}}
  
    \nc{\barsigma}{{\bar \sigma}}

      \nc{\barB}{{\bar  B}}

      \newcommand{\bard}{{\bar  d}}

    \newcommand{\barG}{{\bar G}}

     \newcommand{\barT}{{\bar T}}

  \newcommand{\resp}{{resp.\,}}

  \newcommand{\Spec}{{\rm Spec}}

   \newcommand{\op}{{\rm op}}

   \nc\blank{-}

      \newcommand{\sharpF}{{F^\sharp}}

\nc\diag{{\rm diag}}

\nc{\barpi}{{\bar \pi}}

\nc{\barC}{{\bar C}}

\nc{\barLambda}{{\bar \Lambda}}

\nc{\hatR}{{\hat R}}

\nc{\barphi}{{\bar \phi}}
\renewcommand{\vert}{{\,|\,}}
\nc{\hatL}{{\hat L}}

\nc{\barE}{{\bar   E}}
\nc{\D}{{\mathcal D}}
\nc{\E}{{\mathcal E}}
\nc{\F}{{\mathcal F}}
\nc{\FF}{{\mathcal F}}
\nc{\I}{{\mathcal I}}

\nc{\calP}{{\mathcal P}}
\nc{\even}{{\rm e}}
\nc{\ep}{\epsilon}
\nc{\odd}{{\rm o}}
\nc{\Coker}{{\rm Coker}}
\nc{\olE}{{\overline E}}
\nc{\indBG}{{\rm ind}_B^G\,}
\nc{\indHG}{{\rm ind}_H^G\,}

\nc{\que}{{\mathbb Q}}
\nc{\barlambda}{{\bar\lambda}}
\nc{\barmu}{{\bar\mu}}
\nc{\barnu}{{\bar\nu}}
\nc{\bartau}{{\bar\tau}}
\nc{\barm}{{\overline  m}}
\nc{\divind}{{\rm div.ind}}
\nc{\tl}{{\tilde{\lambda}}}
\nc{\dar}{\downarrow}

\nc{\barnabla}{{\overline \nabla}}

\nc{\en}{{\mathbb N}}
\nc{\eno}{{\mathbb N}_0}

\renewcommand{\det}{{\rm det}}
\nc{\Sym}{{\rm Sym}}
\nc{\Symm}{{\rm Sym}}

\newcommand{\q}{\quad}
\newcommand{\de}{\delta}

\renewcommand{\mod}{{\rm mod}}
\newcommand{\iso}{\cong}

\newcommand{\bs}{\bigskip}

\renewcommand{\vert}{\,|\,}

\renewcommand{\sgn}{{\rm sgn}}

\xyoption{all}

\renewcommand{\vert}{\,|\,}
 \newcommand{\tbw}{\textstyle\bigwedge}
\newcommand{\zed}{{\mathbb Z}}
\newcommand{\Ext}{{\rm Ext}}
\newcommand{\End}{{\rm End}}
\newcommand{\Hom}{{\rm Hom}}
\newcommand{\cf}{{\rm cf}}

\renewcommand{\mod}{{\rm mod}}

\nc{\barb}{{\overline b}}

\newcommand{\m}{{\mathfrak m}}

\renewcommand{\mod}{{\rm{mod}}}

\nc{\geom}{{\rm geom}}
\nc{\rep}{{\rm rep}}

\newtheorem{definition}{Definition}[section]
\newtheorem{proposition}[definition]{Proposition}
\newtheorem{theorem}[definition]{Theorem}
\newtheorem{lemma}[definition]{Lemma}
\newtheorem{corollary}[definition]{Corollary}

\newtheorem{remark}[definition]{Remark}

\begin{document}

\newpage

%section1.tex

\cl{\bf On the conjugation action for quantum general linear groups}

\bigskip

\cl{Stephen  Donkin}

\cl{\it Department of Mathematics }

\cl{\it University of York }

\cl{\it Heslington}

\cl{\it York YO10 5DD,  England UK}

\bigskip

\cl{5 September  2022}

\bs\bs\bs

\bf Abstract  \q \rm  We consider the conjugation action of a quantum group  $G$ over an arbitrary field $k$. In particular we consider the coordinate algebra $k[G(n)]$  of the quantised general linear group $G(n)$, at an arbitrary non-zero parameter $q$, as a $G(n)$-module  and give analogues of results of Kostant and Richardson.  We also consider the case in which $G$ is a product of quantum general linear groups and  the problem of describing the conjugation invariants of $k[G]$ for the action of a quantum subgroup. This is approached  via finite dimensional sub-coalgebras of $k[G]$ 
and  the theory of quasi-hereditary algebras.

\bs

\section{Introduction}

\q\rm Domokos and Lenagan, \cite{DLA},  show that the algebra of polynomial invariants for the conjugation action of a  quantum general linear group  on quantum matrices of degree $n$, over a field of characteristic $0$ at parameter $q$ not a root of unity, is a commutative  polynomial algebra in $n$ variables. Further results  are given in \cite{DLB}. 

\q We consider this situation and also other aspects of the
conjugation representation of a quantum linear  group, as constructed in \cite{DD}. We prove that the algebra of polynomial invariants, for arbitrary $k$ and $0\neq q\in k$, is a commutative  polynomial algebra. We  consider the full algebra of invariants $C(G)$ of the coordinate algebra of a quantum general linear group $G$ and describe, Theorem 7.4,  a $(C(G),G)$-module filtration of the coordinate algebra $k[G]$. This is an analogue of the main result of  \cite{DonkinO} (which is itself an analogue of results of Richardson, \cite[Theorem A]{Richardson} and Kostant, \cite[Theorem 0.13]{Kostant}).  We treat only the case in which $G$ is a quantum general linear group, though it is clear that many of the arguments and results are valid also  for the coordinate algebras of quantized enveloping algebras of semisimple complex Lie algebras.

    \q One of the basic  features of Steinberg's theory of conjugacy classes of a semisimple, simply connected group $G$   (see, for example,   \cite{Steinberg})  is the description of the algebra of class functions $C(G)$ as the subalgebra of the coordinate algebra of $G$ spanned  by the trace functions on rational $G$-modules. In \cite{DonkinQ}, we extended this to give a description of the algebra $C(G,H)$ of regular functions  on $G$ that are class functions relative to a subgroup $H$ which enjoys a certain property which we call saturation. We give, as an investment for the future,  the  $q$ analogue of the more general result concerning the algebra of relative invariants indicated above. The argument that we give, in order to give a description of $C(G,H)$ in terms of trace functions, is a local one, using the structure of finite dimensional subcoalgebras of the coordinate algebra of $G$. With accessibility to a wider audience in mind, and for the sake of variety, we have
  chosen to present this by first dealing with the analogous properties for finite dimensional algebras and dualizing. Sections  2 and 3 are   thus devoted to a version for finite dimensional algebras (satisfying certain conditions) of the general result on $C(G,H)$, \cite[Section 1, Theorem]{DonkinQ}, with particular emphasis on quasi-hereditary algebras in Section 4.  In Section 5 we transfer to the coalgebra setting.   We apply the techniques developed so far to give, in Section 6, a description of the algebra of relative invariants $C(G,H)$ in terms of shifted trace functions. In Section 7 we concentrate on the case $G=H$ and give the results on $(C(G),G)$-module filtrations already mentioned.

   \bs\bs\bs\bs

  %---------------------------------------------------------

%section2.tex

\section{Constrained algebras}

\q Let $k$ be a field. We write $\dim_k V$, or simply $\dim V$, for the dimension of a finite dimensional vector space $V$ over $k$. 

\q Let $S$ be a finite dimensional $k$-algebra which we assume to be   split (or Schurian), i.e., that  $\End_S(L)=k$, for every simple $S$-module $L$. 

\q  We write $\mod(S)$ for the category of finite dimensional left $S$-modules and write $V\in\mod(S)$ to indicate that $V$ is a finite dimensional left $S$-module.  

An $S$-module will be assumed to be a left module and finite dimensional if no other indication is given.

\q We write $J(S)$ for the Jacobson radical of $S$ and write $[S,S]$ for the $k$-span of all commutators $[x,y]=xy-yx$, $x,y\in S$.

\begin{definition} We say that $S$ is (radically) constrained if $J(S)\subseteq  [S,S]$. 
\end{definition}

\q Note that if $S$ is a constrained algebra  and $I$ is an ideal of $S$ then 
$$J(S/I)=(J(S)+I)/I\subseteq ([S,S]+I)/I =  [S/I,S/I]$$
and so:
\bs

(1)\sl \q A factor algebra of a constrained algebra is constrained. 

\bs\rm

\q We write $V^*$ for the linear dual of a vector space $V$. For a subspaces  $H$ of $V$ and $L$ of $V^*$ we write $H^\perp$ and $L^\perp$  for the perpendicular subspaces of $V^*$ and $V$.   Thus  
$$H^\perp=\{\alpha\in V^*\vert \alpha(v)=0 \hbox{ for all } v\in H\}$$
 and  
 $$L^\perp=\{v\in V\vert \alpha(v)=0 \hbox{ for all } \alpha\in H\}.$$

\q For an $S$-module  $V$ and $s\in S$, we write $\trace(s,V)$ for the trace of $\rho(s)$, where $\rho:S\to\End_k(V)$ is the representation afforded by $V$. We have  the natural character $\chi_V\in S^*$ given by $\chi_V(s)=\trace(s,V)$, $s\in S$.  

\q We write $C(S)$ for $[S,S]^\perp$ and $C_0(S)$ for the span of all  natural characters  $\chi_V$, $V\in\mod(S)$.  Note that an element $f$  of $C_0(S)$ vanishes on nilpotent elements (since the trace of a nilpotent endomorphism is $0$)  and hence $f$  vanishes on $J(S)$. Moreover, all elements of  $C_0(S)$ vanish on $[S,S]$ (since $\trace(xy,V)=\trace(yx,V)$, for   $x,y\in S$) and so $C_0(S)\subseteq C(S)$.   The space  $C_0(S)$ has $k$-basis $\chi_1,\ldots,\chi_n$, where $\chi_r=\chi_{V_r}$, $1\leq r\leq n$, and $V_1,\ldots,V_n$ is a complete set of pairwise non-isomorphic simple left $S$-modules. In particular $C_0(S)$ has dimension $l(S)$, the number of isomorphism types of simple left $S$-modules. 
Thus, writing $\Grot(S)$ for the Grothendieck group of $\mod(G)$,  defined by short exact sequences:

\bs

(2) \sl\q We have a $k$-linear isomorphism  $k\otimes_\zed \Grot(S)\to C_0(S)$, taking $1\otimes [V]$ to $\chi_V$ (where $[V]\in \Grot(S)$ denotes the class of $V\in\mod(S)$).

\rm

\begin{lemma} The following are equivalent:

(i) $S$ is constrained;

(ii) $\dim S/[S,S]= l(S)$;

(iii)  $\dim S/[S,S]\leq  l(S)$;

(iv) $[S,S]=C_0(S)^\perp$.
\end{lemma}

\begin{proof}  (i)\implies (ii). From $J(S)\subseteq [S,S]$ we get   \\ $[S/J(S),S/J(S)]=[S,S]/J(S)$. However, $S/J(S)$ is a direct sum of matrix algebras and it follows that 
$[S/J(S),S/J(S)]$ has codimension $l(S)$ in $S/J(S)$. We obtain 
$$\dim\,  [S/J(S),S/J(S)]+l(S)=\dim S/J(S)$$
 so that 
$$\dim [S,S]- \dim J(S)+l(S)=\dim S-\dim J(S)$$
 and therefore $\dim [S,S]+l(S)=\dim S$ and $\dim S/[S,S]= l(S)$. 

(ii)\implies (iii). Clear.

(iii)\implies (iv). Since every element of $C_0(S)$ is zero on  $[S,S]$ we have $[S,S]\subseteq  C_0(S)^\perp$. Hence  $C_0(S)$ embeds in $(S/[S,S])^*$ and so
$$\dim S/[S,S]\geq \dim C_0(S)=l(S)$$
and so $\dim S/[S,S]=l(S)$. Thus $\dim S/[S,S]=\dim C_0(S)$ and $[S,S]\subseteq C_0(S)$ so that $[S,S]=C_0(S)^\perp$.

(iv)\implies (i) We have $J(S)\subseteq  C_0(S)^\perp=[S,S]$. 
\end{proof}

\bs

(3) \sl\q From  (2) we get that if $S$ is constrained and $V_1,\ldots,V_n$ is a complete set of pairwise non-isomoprhic irreducible $S$-module then their natural characters $\chi_1,\ldots,\chi_n$ form a basis of $C(S)$.

\bs\rm

\q Given an $S$-bimodule $M$ we write $[S,M]$ for the subspace of $M$ spanned by all element $sm-ms$, with $s\in S$, $m\in M$. We write $H_i(S,M)$ for the degree $i$ Hochschild homology. Thus $H_0(S,M)=M/[S,M]$ and from Lemma 2.2 we have:
\bs

(4) \sl\q $S$ is constrained if and only if $H_0(S,S)$ has dimension at most (equivalently equal to) $l(S)$.

\bs

\q\rm Moreover, from the Morita invariance property for Hochschild homology (see e.g., \cite[Theorem 9.5.6]{Weibel})  we have:
\bs

(5)\sl\q\q Suppose that $S'$ is a finite dimensional $k$-algebra Morita equivalent to $S$. Then $S'$ is split  and $S'$ is constrained if and only if $S$ is constrained.
\bs\rm

(One may also see this directly by considering $R/[R,R]$ for an algebra $R$ of the form $\End_S(S\oplus Se)$, where $e\in S$ is an idempotent.)

\bs

(6) \q\sl If $K$ is an extension field of $k$ then the algebra $S_K=K\otimes_k S$, obtained by base change, is constrained if and only if $S$ is constrained.

\bs\rm

\q This follows from the fact that $l(S)=l(S_K)$ and the natural map $K\otimes_k [S,S]\to [S_K,S_K]$ is an isomorphism.

\bs

We write $\Cent(R)$ for the centre of a ring $R$.

\begin{lemma} (i) If $S$ is constrained and $R$ is a symmetric, factor algebra of $S$ then $R$ is semisimple.

(ii) If $S$ is a symmetric algebra that is not semisimple  then the dimension of the centre  $\Cent(S)$ of $S$ is greater than $l(S)$.
\end{lemma}

\begin{proof} (i)  If $I$ is an ideal of $S$ then $S/I$ is constrained, by (1). Thus it suffices to show that if $S$ is symmetric (and constrained) then it is semisimple. Let $(\,,\,)$ be a form on $S$ that  makes $S$ a symmetric algebra. For a subspace $V$ of $S$ we now set
$$V^\perp=\{x\in S\vert (x,y)=0 \hbox{ for all } y\in V\}.$$
 If $V$ is an ideal then $V^\perp$ is too. The centre $\Cent(S)$ of $S$ is equal to $[S,S]^\perp$.  Now $J(S)\subseteq  [S,S]$ gives $\Cent(S)\subseteq  J(S)^\perp$. Thus  $J(S)^\perp$ is an ideal containing the identity element of $S$ and hence $J(S)^\perp=S$ and $J(S)=0$.

(ii) We have $\dim S/[S,S]> l(S)$, by Lemma 2.2, i.e.,  
$$\dim  \Cent(S)=\dim\, [S,S]^\perp>l(S).$$.  
\end{proof}

\begin{remark} One of the main result  of the paper \cite{LenzingH}, by H. Lenzing, is that  in a finite dimensional algebra of finite global dimension, every nilpotent element is a sum of commutators, \cite[Satz 5]{LenzingH},  or, expressed in our language, that such an algebra is constrained.  
In particular this gives that  a quasihereditary algebra is constrained  (we give a more direct argument for this in Remark 4.1 below). 

\q As noted by Igusa, \cite[Introduction]{Igusa},  Lenzing's incisive result and the elementary observations above give a very short  proof of the no loops conjecture for $S$ a finite dimensional split algebra $S$.  Suppose for a contradiction $S$ is a split  $k$-algebra of finite global dimension and $\Ext^1_S(L,L)\neq 0$ for some non-zero simple module $L$. The algebra  $S$ is  Morita equivalent to a basic algebra $R$ over $k$. Now $R$ also  has finite global dimension and has a one dimensional module $M$ such that $\Ext^1_R(M,M)\neq 0$.  Let $0\to M\to E\to M\to 0$ be a non-split short exact sequence of $R$-modules. The representation $\rho:R\to\End_k(E)$ has image isomorphic to the algebra $\{\left(\begin{matrix}a&b\cr
0&a\end{matrix}\right)\vert a,b\in k\}$, i.e., isomorphic to $A=k[t]/(t^2)$. But we have $ J(A)\neq 0$ and $[A,A]=0$ so this algebra is not constrained. But then neither is $R$ by, (1) above, and $R$ has finite global dimension so we have a contradiction.
\end{remark}

\begin{remark} Suppose, for the moment, that $k$ is an algebraically closed field of characteristic $p>0$ and let ${\mathfrak g}$ be the Lie algebra of a connected, semisimple, simply connected,  algebraic group $G$ over $k$. In \cite{Rumynin}, p370, the question is raised as to whether the dimension of the centre of the restricted enveloping algebra $u(\gothg)$ is equal to $p^r$, where $r$ is the rank of $G$, However, the algebra is symmetric, not semisimple and, by a theorem of Curtis (see e.g., \cite[II, 3.10 Proposition]{RAG} ) the number of isomorphism classes of simple modules is $p^r$ and so, by Lemma 2.3 (ii) the centre  has dimension greated than $p^r$.  A similar remark applies more generally to the algebra $u_n(\gothg)$ of distributions of the $n$th Frobenius kernel of $G$.  A more detailed, geometric, investigation into the  centre, in the analogous case of the small quantum group, is to be found in \cite{BL}.
\end{remark}

\begin{remark}  We continue with the notation of Remark 2.5. Then $G$ is split over the prime field via its structure as a Chevalley group.  For $n\geq 1$, let $G(p^n)$ be the group of elements defined over the field of $p^n$ elements.  The algebras $kG(p^n)$ and $u_n$ are symmetric. 
In \cite{MDV},  the existence of interesting  quasi-hereditary factor algebras of the group algebra $kG(p^n)$ and of $u_n$ is demonstrated.  By contrast the opposite problem of finding symmetric algebra factor algebras of a quasihereditary algebra has only trivial solutions. More precisely, a symmetric factor algebra of a quasihereditary algebra must be semisimple, by Lemma 2.3(i).  In particular if $H$ is a finite group such that $kH$ occurs as a factor  algebra of a quasihereditary algebra (or more generally a finite dimensional algebra of finite global dimension) then $H$ has order prime to $p$.
\end{remark}

 \bs\bs\bs\bs

%-----------------------------

%section3.tex

\section{Relative Constraint} 

\q We now consider a relative version of constraint.  Let $R$ be a subalgebra of the finite dimensional split  $k$-algebra $S$.   We define $C(S,R)$ to be the subspace of the linear dual $S^*$ of $S$ consisting of all elements  $f$ such that $f(rs)=f(sr)$, for all $r\in R$, $s\in S$.  Thus $C(S,R)$ consists of all elements that vanish on the space $[R,S]$, spanned by all commutators $rs-sr$, $r\in R$, $s\in S$, and
$$\dim C(S,R)=\dim\, (S/[R,S])^*=\dim H_0(R,S).$$

\q Some of the elements of $C(S,R)$ may be accounted for representation theoretically (analogously to the trace functions on $S$ considered in Section 2) in  the following way. Let $V$ be a finite dimensional left $S$-module affording the representation $\rho:S\to \End_k(V)$ and let $\theta$ be an element of $\End_R(V)$. We write $\chi_\theta$ or $\chi_{\theta,V}$ for the \lq\lq twisted trace function", i.e., the element of $S^*$ defined by $\chi_\theta(s)=\trace(\theta\rho(s))$, $s\in S$.   It follows from the fact that $\theta$ commutes with the action of $R$ that $\chi_\theta\in C(S,R)$.  The set 
$$C_0(S,R;V)=\{\chi_\theta\vert \theta\in\End_R(V)\}$$
 is a subspace of $C(S,R)$. We define $C_0(S,R)\subseteq  C(S,R)$ to be the sum of all spaces $C_0(S,R;V)$, as $V$ ranges over all finite dimensional $S$-modules. Note that $C(S,R;V)=C(S,R;W)$ if $V$ and $W$ are isomorphic $S$-modules. Moreover, for finite dimensional $S$-modules $V,W$,  we have 
 $$C(S,R;V\oplus W)=C(S,R;V)+C(S,R;W).$$   It follows that $C_0(S,R)$ is the sum of the spaces $C_0(S,R;V)$, as $V$ ranges over all indecomposable $S$-modules (cf. \cite[Section 2]{DonkinRelative}).

\begin{definition} We say that $R$ is a constrained subalgebra of $S$, or that the pair $(S,R)$ is constrained,  if $C_0(S,R)=C(S,R)$. 
\end{definition}

\begin{remark} Note that $R$ is a constrained subalgebra if and only if \\
$\dim H_0(R,S)= \dim C(S,R)$ (equivalently 
$\dim H_0(R,S)\leq \dim C_0(S,R)$), by the argument of Lemma 2.2.
\end{remark}

\begin{remark} We note  that  $S$ is (radically) constrained if and only if the pair $(S,S)$ is constrained. 

\q If $V\in\mod(S)$ then $\chi_V=\chi_\theta$, where $\theta:V\to V$ is the identity map. Hence we have $C_0(S)\subseteq  C_0(S,S)$. If $S$ is constrained then $\dim C_0(S,S)\geq \dim C_0(S)=\dim H_0(S,S)$ and so $(S,S)$ is constrained. 

\q Suppose conversely that $(S,S)$ is constrained.  To show that $S$ is constrained we may, by Section 2, (6), assume that $k$ is algebraically closed. 

\q Let $V$ be an indecomposable finite dimensional left $S$-module affording the representation $\pi:S\to\End_k(V)$. For $\theta\in \End_S(V)$ we have $\theta=\lambda\,\id_V+\nu$,  for some $\lambda\in k$ and some nilpotent endomorphism $\nu$ (where $\id_V:V\to V$ is the identity map). But then, for $s\in S$, we have
$$\chi_\theta(s)=\trace(\pi(s)\theta)=\trace(\lambda\pi(s))+\trace(\pi(s)\nu).$$
Since $\pi(s)\nu$ is nilpotent, its trace is zero and so we have $\chi_\theta=\lambda\chi_V$ and $\chi_\theta\in C_0(S)$. Now an element of $C_0(S,S)$ is a sum of elements $\chi_\theta$, with $\theta\in\End_S(V)$ and $V$ indecomposable. Hence we have $C_0(S,S)=C_0(S)$.  Hence we have $\dim C_0(S)=\dim C_0(S,S)\geq \dim H_0(S,S)$ and $S$ is constrained.
\end{remark}

\begin{remark}
Let $M\in\mod(R)$ and let $\pi:R\to\End_k(M)$ be the representation afforded by $M$.  Then $\End_k(M)$ is an $R$-bimodule, with $r\phi=\pi(r)\phi$, $\phi r= \phi \pi(r)$, for $r\in R$, $\phi\in \End_k(M)$.  Now  $\End_k(M)$ has the non-singular symmetric bilinear form defined by $(\phi,\psi)=\trace(\phi\psi)$ and it is easy  to check that $\End_R(M)=[R,\End_k(M)]^\perp$. We shall need this in the proof of the following result. 
\end{remark}

\begin{proposition} Let $R$ be a subalgebra of $S$.

(i) For $M\in\mod(S)$ we have $C(S,R)=C_0(S,R;M)$ if and only if the representation $\rho:S\to \End_k(M)$ induces an injection \\
$H_0(R,S)\to H_0(R,\End_k(M))$.

(ii)  $R$ is a constrained subalgebra if and only if there exists $M\in\mod(S)$ such that the induced map $H_0(R,S)\to H_0(R,\End_k(M))$ is injective.
\end{proposition}

\begin{proof}  (i) For a subspace $V$ of $S^*$ we now set
$$V'=\{x\in S \vert \alpha(x)=0 \hbox { for all } \alpha\in V\}.$$

\q Let $\phi: H_0(R,S)\to H_0(R,\End_k(M))$  be the natural map. For   $x=s+[R,S]$, $s\in S$, we have $\phi(x)=\rho(s)+[R,\End_k(M)]$.  Now we have $\phi(x)=0$ if and only if $\rho(s)\in [R,\End_k(M)]=\End_R(M)^\perp$.  Thus $\phi(x)=0$ if and only if $\trace(\rho(s)\theta)=0$ for all $\theta\in \End_R(M)$, i.e., if and only if $s\in C(S,R;M)'$. 

\q Hence  $\phi$ is injective if and only if an element $s\in S$ belongs to $[R,S]$ if and only if it belongs to $C_0(S,R;M)'$. But $[R,S]=C(S,R)'$ so that $\phi$ is injective if and only if $C(S,R)=C_0(S,R;M)$.

(ii) If such an $S$-module $M$  exists then $R$ is a constrained subalgebra by (i).

\q Suppose now that $R$ is a constrained subalgebra.   We choose a basis $\chi_{\theta_1,M_1},\ldots,\chi_{\theta_n,M_n}$ of $C(S,R)$, with $M_i \in \mod(S)$,
$\theta_i\in \End_R(M_i)$, for $1\leq i\leq n$.  Putting $M=M_1\oplus\cdots\oplus M_n$, we have that each $\chi_{\theta_i,V_i}$ belongs to $C(S,R;M)$ and so we have $C(S,R)=C_0(S,R;M)$.

\q Let $\pi:S\to \End_k(M)$ be the representation afforded by $M$. We claim that $H_0(R,S)\to H_0(R,\End_k(V))$ is injective.

\q Suppose, for a contradiction,  we have. $s\in S\backslash [R,S]$ with   $s+[R,S]$  mapping to $0$.  Thus we have $\pi(s)\in [R,\End_k(M)]$, i.e., 
$$\pi(s)=\sum_{i=1}^t\pi(r_i)\alpha_i-\sum_{i=1}^t\alpha_i\pi(r_i) \eqno{(*)}$$
for some $r_1,\ldots,r_t\in R$ and $\alpha_1,\ldots,\alpha_t\in\End_k(M)$.  Since $s\not\in [R,S]$ we have  $f(s)\neq 0$ for some $f\in C(S,R)$.  However, $f=\chi_{\theta,M}$ for some $\theta\in \End_R(M)$, and we have $\trace(\theta\pi(s))\neq 0$, i.e.,
$$\sum_{i=1}^t\trace(\theta\pi(r_i)\alpha_i)-\sum_{i=1}^t\trace(\alpha_i\pi(r_i)\theta)\neq 0.$$
But, for $1\leq i\leq t$, we have
\begin{align*}\trace(\theta\pi(r_i)\alpha_i)&=\trace(\pi(r_i)\theta\alpha_i)=\trace(\theta\alpha_i\pi(r_i)\cr
&=\trace(\alpha_i\pi(r_i)\theta))
\end{align*}
since $\theta$ commutes with each $\pi(r_i)$,  so that, by (*), $\trace(\theta\pi(s))=0$, a contradiction.
\end{proof}

\begin{remark} Let $K$ be a field extension of $k$. If $M\in \mod(S)$ is such that $H_0(R,S)\to H_0(S,\End_k(M))$, then by base change we obtain that $H_0(R_K,S_K)\to  H_0(R_K,\End_K(M_K))$ is injective (where $M_K=K\otimes_k M\in \mod(M_K)$).  Hence if $R$ is a constrained $k$-subalgebra of $S$ then $R_K$ is a constrained subalgebra of $S_K$.
\end{remark}

\begin{corollary}  If $R$ or $S$ is semisimple then $R$ is a constrained subalgebra of $S$.
\end{corollary}

\begin{proof}   We take $M$ to be any faithful 
finite dimensional left $S$-module. The  injection of $R$-bimodules $S\to \End_k(M)$ is split 
(if $S$ is semisimple  because every injection of $S$-bimodules is split and if $R$ is semisimple because every injection of $R$-bimodules is split) and so the induced map 
$H_0(R,S)\to H_0(R,\End_k(M))$ is injective. 
\end{proof}

\begin{remark} One may ask whether there is a relative version of Lenzing's result (described in  Remark 2.4),  i.e.,  whether one can give suitable homological conditions which imply that a subalgebra $R$ is a constrained subalgebra of a  finite dimensional algebra $S$.
\end{remark}

\bs\bs\bs\bs

%-------------------------------------------------------

%section4.tex

\section{Quasi-hereditary algebras}

\q Let $S$ be  finite dimensional algebra over the field $k$. By abuse of notation we shall use the expression \lq\lq $S$ is a quasi-hereditary algebra" to indicate that, 
with respect to some labelling $L(\lambda)$, $\lambda\in\Lambda$, of a complete set of pairwise inequivalent
irreducible left $S$-modules, and some partial order $\leq$ on $\Lambda$, the category of finite dimensional left
$S$-modules is a highest  weight category (see for example \cite[Appendix]{DonkinAA}).

\q  We  now assume now $S$ is quasi-hereditary and that  $L(\lambda)$, $\lambda\in\Lambda$ is a labelling of the irreducible left $S$-modules and $\leq$
a partial order with respect to which $\mod(S)$ is a highest  weight category.
 For $\lambda\in\Lambda$ let 
$\Delta(\lambda)$ (\resp $\nabla(\lambda)$, \resp $M(\lambda)$) be the corresponding standard 
module (\resp costandard module, \resp indecomposable tilting module). 

\q  
\q We write $S^\op$ for the opposite algebra of $S$. For a 
finite dimensional left (\resp right) $S$-module $X$ we regard the dual space 
$X^*=\Hom_k(X,k)$ as a right (\resp left) module in the natural way. 
Then $L(\lambda)^*$, $\lambda\in \Lambda$ is a 
complete set of pairwise non-isomorphic right $S$-modules. Moreover, $S^\op$ is a 
quasihereditary for the labelling $L(\lambda)^*$, $\lambda\in\Lambda$ 
(with respect to the existing order on $\Lambda$).  For $\lambda\in\Lambda$, 
the corresponding standard module (\resp costandard module, \resp  
indecomposable  tilting module) is 
$\nabla(\lambda)^*$ (\resp $\Delta(\lambda)^*$, \resp $M(\lambda)^*$). 

\q By a standard (\resp costandard) filtration of $X\in\mod(S)$ we mean a filtration $0=X_0\subseteq X_1\subseteq \cdots. \subseteq X_n=X$ such
that  for each $1\leq i\leq n$, quotient $X_i/X_{i-1}$ is either $0$ or a  standard (\resp costandard) module. For $X\in\mod(S)$, we write 
$X\in\F(\Delta)$ (\resp $X\in\F(\nabla)$) to indicate that $X$ admits a
standard filtration (\resp $X$ admits a costandard filtration).

\q If $X\in\mod(S)$ admits a standard filtration (\resp costandard) filtration then,  for $\lambda\in\Lambda$, the
dimension of $\Hom_S(X,\nabla(\lambda))$ (\resp the dimension of  $\Hom_S(\Delta(\lambda),X)$) is the cardinality of $\{1\leq i\leq
n\vert X_i/X_{i-1}\iso \Delta(\lambda)\}$ (\resp 
$\{1\leq i\leq n\vert X_i/X_{i-1}\iso \nabla(\lambda)\}$) where
$0=X_0\subseteq X_1\subseteq \ldots\subseteq X_n=X$ is any standard (\resp costandard) filtration of $X$. We denote this integer by
$(X:\Delta(\lambda))$ (\resp $(X:\nabla(\lambda))$).  We refer the reader to \cite[Appendix]{DonkinAA} for basic facts about quasi-hereditary algebras not explicitly stated here.

\q The tensor product algebra 
$S\otimes S^\op$ has complete set of simple modules  $L(\lambda,\mu)$, $(\lambda,\mu)\in\Lambda\times\Lambda$, 
where $L(\lambda,\mu)=L(\lambda)\otimes L(\mu)^*$.  We give 
$\Lambda\times\Lambda$ the partial order: $(\lambda,\mu)\leq (\lambda',\mu')$ if (and only if) 
$\lambda\leq\lambda'$ and $\mu\leq \mu'$.  Then  $\mod(S\otimes S^\op)$ is quasihereditary 
(i.e., the category of left $S\otimes S^\op$-module is a highest weight category) with respect to this order. 
Moreover, we have $\nabla(\lambda,\mu)=\nabla(\lambda)\otimes \Delta(\mu)^*$, $\Delta(\lambda,\mu)=\Delta(\lambda)\otimes\nabla(\mu)^*$ and \\
$M(\lambda,\mu)=M(\lambda)\otimes M(\mu)^*$, for $\lambda,\mu\in\Lambda$. We identify the categories of $S$-bimodules and left $S\otimes S^\op$-modules in the usual way.

\begin{remark}From this one quickly gets a sequence of ideals as in the usual definition of quasi-hereditary, and that $S$ is constrained.
Consider $S$ as an $S$-bimodule. Then, for $i\geq 0$,
and $\lambda,\mu\in \Lambda$, $i\geq 0$, we have 

$$\Ext_{S\otimes S^\op}^i(S,\nabla(\lambda,\mu))=\Ext_{S\otimes S^\op}^i(S,\nabla(\lambda)\otimes \Delta(\mu)^*)=\Ext^i_S(\Delta(\mu),\nabla(\lambda))$$
by \cite[Lemma 9.1.3 and Lemma 9.1.9]{Weibel}, and this is $k$ if $\lambda=\mu$ and $i=0$ and $0$ otherwise.  Hence, by \cite[Proposition A 2.2(iii)]{DonkinAA},  the regular $S$ bimodule has a $\Delta$-filtration and 
$$(S:\Delta(\lambda,\mu))=\dim \Hom_{S\otimes S^\op} (S,\nabla(\lambda,\mu))$$
is one if $\lambda=\mu$ and $0$ otherwise.  Thus, if $\Lambda$ consists of the distinct elements $\lambda_1,\ldots,\lambda_n$ ordered so that $\lambda_i>\lambda_j$ implies $i<j$,  then $S$ has a bimodule filtratiion (i.e.,a sequence of ideals) $S=I_0\supset  I_1\supset \cdots \supset I_{n+1}=0$ with $I_j/I_{j+1}\cong \Delta(\lambda_j)\otimes \nabla(\lambda_j)^*$, for $1\leq j\leq n$. 

\q Let $M$ be a finite dimensional $S$-bimodule. Recall the we have the  Hochschild cohomology spaces  $H^i(S,M)$, $i\geq 0$, and 
$$H^0(S,M)=\{m\in M \vert sm=ms \hbox{ for all } s\in S\}.$$
The  dual space $M^*$ is naturally an $S$-bimodule with actions given by $s\alpha(m)=\alpha(ms)$ and $\alpha s (m)=\alpha(sm)$, $\alpha\in M^*$, $s\in S$, $m\in M$.  Moreover the natural map $M^*\times M\to k$ induces a vector space isomorphism $H^0(S,M^*)\to H_0(S,M)^*$.  Taking  $M$ to be the dual $S^*$ of the regular bimodule we have $C(S)=H^0(S,S^*)$.  Now $S^*$ has a bimodule filtration with sections $\Delta(\lambda,\lambda)^*=\nabla(\lambda)\otimes \Delta(\lambda)^*$, $\lambda\in \Lambda$ and 
$$H^i(S,\nabla(\lambda)\otimes \Delta(\lambda)^*)=\Ext^i_S(\Delta(\lambda),\nabla(\lambda))$$
by \cite[Lemma 91.3 and Lemma 9.1.9]{Weibel}, and this is $k$ for $i=0$ and $0$ for $i=1$ so we get
$\dim H^i(S,S^*)= |\Lambda|=l(S)$, i.e., $\dim C(S)=\dim C_0(S)$ so that $C(S)=C_0(S)$ and $S$ is constrained. 
\end{remark}

\begin{definition} We shall say that  a subalgebra $R$ iof $S$ is saturated  if, for every  $\lambda\in\Lambda$, the
functor $\Hom_R(\Delta(\lambda),\blank)$ is exact on short exact sequences of finite dimensional $S$-modules which admit a costandard filtration.
\end{definition}

\q  We shall give a numerical criterion for saturation.  For $\lambda,\mu\in\Lambda$ we set 
$$a_R(\lambda,\mu)=\dim\Hom_R(\Delta(\lambda),\nabla(\mu))=\dim
H_0(R,\Delta(\lambda)\otimes \nabla(\mu)^*).$$  
For $X\in\F(\Delta)$, $Y\in \F(\nabla)$
we set
$$a_R(X,Y)=\sum_{\lambda,\mu\in \Lambda}(X:\Delta(\lambda))(Y:\nabla(\mu))a_R(\lambda,\mu).$$
  By left exactness of the Hom functor we have  $\dim\Hom_R(X,Y)\leq a_R(X,Y)$ and we say
that pair $(X,Y)\in \F(\Delta)\times\F(\nabla)$ is {\it well adapted for $R$} (or just well adapted) if
equality holds. Note  (again by left exactness) that  $R$ is saturated if and only if $(X,Y)$ is well adapted for
every $X\in\F(\Delta)$, $Y\in\F(\nabla)$. Moreover, if $0\to X_1\to X\to X_2\to 0$ is a short exact sequence in
$\F(\Delta)$ and $Y\in\F(\nabla)$ and $(X,Y)$ is well adapted then both pairs $(X_1,Y)$ and $(X_2,Y)$ are well
adapted. Similarly if $X\in\F(\Delta)$ and $0\to Y_1\to Y\to Y_2\to 0$ is a short exact sequence in $\F(\nabla)$ and
$(X,Y)$ is well adapted then both pairs $(X,Y_1)$ and $(X,Y_2)$ are well adapted. Now if $X\in\F(\Delta)$ and
$\phi:P\to X$ is a surjective homomorphism from a projective module then both $P$ and the ${\rm Ker}(\phi)$  belong to
$\F(\Delta)$. Hence $R$ is saturated provided that $(P,Y)$ is well adapted for every projective $P\in\mod(S)$.
Similarly one deduces that indeed $R$ is saturated provided that $(P,I)$ is well adapted for all projective $P\in\mod(S)$
and injective $I\in\mod(S)$.  It follows that $R$ is saturated if and only if $({}_SS,S_S^*)$ is well adapted.
(Here ${}_SS$ denotes the left regular module,  $S_S$ the right regular module,  and $S_S^*$ its dual). In other words,
$R$ is well adapted if and only if $\dim H_0(R,{}_SS\otimes S_S)=\dim \Hom_R({}_SS,S_S^*)$ is equal to
$a_R({}_SS,S_S^*)$. Now we have 
$$({}_SS:\Delta(\lambda))=\dim \Hom_S({}_SS,\nabla(\lambda))=\dim\nabla(\lambda)$$
 and
similarly $(S_S^*:\nabla(\lambda))=\dim\Delta(\lambda)$. Putting all  this  together we have the following.

\begin{proposition}  $R$ is saturated if and only if 
$$\dim H_0(R,{}_SS\otimes
S_S)=\sum_{\lambda,\mu\in \Lambda}\dim\nabla(\lambda)  \dim\Delta(\mu) \dim\Hom_R(\Delta(\lambda),\nabla(\mu)).$$
\end{proposition}

\q By a standard filtration of an $S$ bimodule $X$ we mean a standard filtration of $X$ viewed as a left
$S\otimes S^\op$-module.  Using Remark 4.1 is not difficult to deduce the following in  a similar fashion.

\begin{lemma}  $R$ is saturated if and only if $H_0(R,\blank)$ is exact on short exact sequences of $S$-bimodules 
which admit a standard filtration and if $R$ is saturated then $H_1(R,M)=0$ for every $S$-bimodule which has a standard filtration.
\end{lemma}

%\subsection{}

\q  Recall that a tilting module $M$ for $S$ is said to be \emph{full}  if $M(\lambda)$ occurs as a component of $M$ for all $\lambda\in \Lambda$. We shall need the following.

\begin{proposition}
Let $M$ be a full tilting module. Then the representation $\rho:S\to \End_k(M)$ is injective and  $\Coker(\rho)$  has a standard $S$-bimodule filtration. 
\end{proposition}

\q The fact that $\rho$ is faithful is well known. Indeed,  if $V$ is any module with a standard filtration then
there exists a finite resolution 

$$0\to V\to X_0\to X_1\ldots\to X_n\to 0$$
where $X_0,X_1,\ldots,X_n$  are all (partial) tilting modules (see e.g., the argument of  \cite[Section 1, Theorem]{DonkinI}). In particular $V$ embeds in a tilting module, and hence in a direct sum of copies of $M$.  In the case $V={}_SS$ we get that $X_0$ is faithful and since $X_0$ is a summand of a direct sum of
copies of $M$, we get that $M$ is faithful.

\q To treat the remaining part, we need a lemma.

\begin{lemma}  (i)  Let $\phi:X\to Y$, $\psi:X\to Z$ be $S$-module homomorphisms. Let $W=\Ker(\phi)$ and let $\rho=\psi|_W:W\to Z$. Suppose that $\Coker(\phi)\in\F(\Delta)$, that $\rho$ is injective and that $\Coker(\rho)\in \F(\Delta)$.  Then the map $\chi:X\to Y\oplus Z$, given by $\chi(x)=(-\phi(x),\psi(x))$, $x\in X$, is injective and $\Coker(\chi)\in\F(\Delta)$. 

(ii)  Let $M,N$ be $S$-modules affording the representations $\phi:S\to \End_k(M)$, $\psi:S\to \End_k(N)$. Let $I=\Ker(\phi)$ and let $\rho=\psi\vert_I : I\to \End_k(N)$.  Suppose that $\rho$ is injective and  that $\Coker(\rho)$ and $\Coker(\phi)$ have  standard $S$-bimodule filtrations.  Then the map $\chi:S\to \End_k(M)\bigoplus \End_k(N)$, given by $\chi(s)=(-\phi(s),\psi(s))$, $s\in S$,  is injective and $\Coker(\chi)$ has a standard $S$-bimodule filtration. 
\end{lemma}

\begin{proof} (i) Injectivity is clear. For the second assertion we  proceed as in  \cite[Lemma 2]{DonkinI}.    One may readily check that there is an exact sequence 
$$0\to W\to Z\to\Coker(\chi)\to\Coker(\phi)\to0\eqno{(*)}$$
with maps $\rho:W\to Z$,  and $\alpha:Z\to\Coker(\chi)$ and $\beta:\Coker(\chi)\to\Coker(\phi)$ given by $\alpha(z)=(0,z)+\Im(\chi)$, $z\in Z$, and $\beta((y,z)+\Im(\chi))=y+\Im(\phi)$, $y\in Y$, $z\in Z$. (One may also use the mapping cone construction, \cite[Section 1.5]{Weibel}.)
Thus we get an exact sequence 
$$0\to \Coker(\rho)\to\Coker(\chi)\to\Coker(\phi)\to0.$$
Since 
$\Coker(\rho),\Coker(\phi)\in\F(\Delta)$ we must also have $\Coker(\chi)\in\F(\Delta)$.

(ii) Follows from (i).
\end{proof}

\rm\q We now prove Proposition 4.5.  As already noted,  $\rho$ is injective.   

\q Let $\lambda\in\Lambda$ be a maximal element. Let
$\pi=\Lambda\backslash\{\lambda\}$.   We write $M=M_0\bigoplus M_1$, where $M_0$ is a direct sum of modules of the form
$M(\mu)$, $\mu\in\pi$, and $M_1$ is a direct sum of copies of $M(\lambda)$. Note that the image of
$\rho:S\to\End_k(M)$ lies in $\End_k(M_0)\bigoplus \End_k(M_1)$ and \\
$\End_k(M)/(\End_k(M_0)\bigoplus
\End_k(M_1))$ is isomorphic to $M_0\otimes M_1^*\bigoplus M_0^*\otimes M_1$, which has a standard $S$-bimodule
filtration. Hence it suffices to prove that the cokernel of the  restriction 
$\rho':S\to\End_k(M_0)\bigoplus \End_k(M_1)$ has a standard filtration.

\q Recall, from \cite[Appendix]{DonkinAA}, that we have the functor
$O^\pi:\mod(S)\to\mod(S)$. For $X\in\mod(S)$, $O^\pi(X)$ is minimal among those submodules $Y$ of $X$ such that all
composition factors of $X/Y$ belong to $\{L(\mu)\vert\mu\in\pi\}$. If $f:X\to Y$ is a homomorphism of finite
dimensional $S$-modules then $O^\pi(f)$ is the restriction of $f$.  We list the elements of $\Lambda$ as
$\lambda_1,\lambda_2,\ldots,\lambda_n$ in such a way that $i<j$ whenever $\lambda_i<\lambda_j$ and
$\lambda_n=\lambda$. By  Remark 4.1, there is a chain of ideals $S=I_1 \supset I_2 \supset \cdots I_n \supset I_{n+1}=0$ with
$I_r/I_{r+1}\iso \Delta(\lambda_r)\otimes \nabla(\lambda_r)^*$, for $1\leq r\leq n$. Note that all composition
factors of $S/I_n$ come from $\{L(\mu)\vert\mu\in\pi\}$ and $L(\lambda)$ occurs in every quotient module of $I_n$.
Hence we must have $I_n=O^\pi(S)$. We put $I=I_n=O^\pi(S)$.

\q  Now $S/I$ is a quasihereditary algebra, by \cite[Appendix, Proposition A 3.7]{DonkinAA}, and
$M_0$ is a full tilting module for $S/I$. Hence we may assume, inductively, that the 
 induced map $\rho_0:S/I\to\End_k(M_0)$ is injective and has
cokernel in $\F(\Delta)$.  Hence by the Lemma 4.6(ii),   it is enough to prove that the homomorphism $\sigma:I\to\End_k(M_1)$
(given by $\sigma(s)m=sm$, for $s\in S$, $m\in M$) is injective with cokernel having a standard 
$S\otimes S^\op$-module filtration. Certainly $\sigma$ is injective (since $M=M_0\bigoplus M_1$ is faithful and $I$ acts as  
zero on $M_0$).  We write $M_1=Y_1\bigoplus \cdots\bigoplus Y_m$, where each $Y_i$ is isomorphic to $M(\lambda)$.  The
image of $\sigma$ lies in $\End_k(Y_1)\bigoplus \cdots\bigoplus \End_k(Y_m)$ and 
$$\End_k(M_1)/(\End_k(Y_1)\bigoplus
\cdots\bigoplus \End_k(Y_m))\iso \bigoplus_{1\leq i,j\leq m, i\neq j}\Hom_k(Y_i,Y_j).$$
Moreover, we have $\Hom_k(Y_i,Y_j)\iso M(\lambda)\otimes M(\lambda)^* $, which has a standard $S\otimes
S^\op$-module filtration, $i\neq j$. Thus it is enough to prove that the cokernel of the restriction

$\sigma':I\to \End_k(Y_1)\bigoplus\cdots\bigoplus \End_k(Y_m)$ has a standard $S$-bimodule filtration. 
Furthermore, by the above Lemma and induction on $m$,  it suffices to prove that the 
cokernel of the map $\tau:I\to\End_k(N)$ has a standard $S$-bimodule filtration, where $N=M(\lambda)$ and $\tau(x)n=xn$, for $x\in I$, $n\in N$.

\q Now $N$ contains a submodule $N_1$, say, isomorphic to $\Delta(\lambda)$, and $N/N_1$ belongs to $\pi$. Regarding $\End_k(N)= N\otimes N^*$ as a left $S$-module, we get $O^\pi(\End_k(N))=N_1\otimes N^*$ and hence $\tau$ takes $I=O^\pi(I)$ into $N_1\otimes N^*$.  Moreover, we have $(N\otimes N^*)/N_1\otimes N^*\iso (N/N_1)\otimes N^*$, which has a standard module filtration, as an $S$-bimodule, and hence it suffices to prove that 
the cokernel of the restriction $\tau':I\to N_1\otimes N^*$, i.e., $\tau':I\to \Hom_k(N,N_1)$,  has a standard filtration.  But now, we have a submodule $N_2$ of $N$ such that $N_2\in\F(\nabla)$ and $N/N_2\iso \nabla(\lambda)$. Regarding $\Hom_k(N_1,N)\iso N\otimes N_1^*$ as a right $S$-module, we have that $I$ maps into $O^\pi(N_1\otimes N^*)=N_1\otimes (N/N_2)^*$. By dimensions, we must have that the image of $I$ is $N_1\otimes (N/N_2)^*$. We have  $(N_1\otimes N^*)/(N_1\otimes (N/N_2)^*)\iso N_1\otimes N_2^*$, which has a standard filtration, as an $S$-bimodule, so we are done.

\q By Proposition 3.5 we have the following.

\begin{corollary}  Let $R$ be a saturated subalgebra.  Then $C(S,R)=C_0(S,R,M)$, for any full tilting module $M$ and so $R$ is constrained.  
 In particular $S$ is constrained.
\end{corollary}

\bs\bs\bs\bs

%----------------------------------------

%section5.tex

\section{Constrained coalgebras}

\q We now consider the situation for coalgebras. Let $(A,\de\,\ep)$ be a coalgebra over $k$, which we assume split, i.e., $\End_A(V)=k$ for every irreducible comodule $V$.

\q Let $S$ be
an associative $k$-algebra.  A left $S$-module  $V$ is called locally finite (dimensional) if every cyclic
submodule (and hence every finitely generated submodule) is finite dimensional. If $V$ is locally finite and
$\{v_i\vert i\in I\}$ is a $k$-basis then  the elements $f_{ij}\in S^*$, defined by the equations $sv_i=\sum_{j\in
I}f_{ji}(s)v_j$, for $s\in S$, $i\in I$, are called the coefficient functors of $V$ (relative to this basis). 
The subspace of $S^*$ spanned by all $f_{ij}$, $i,j\in I$, is independent of the choice of basis. We call it the coefficient space of $V$ and denote it by $\cf(V)$.

\q Let $V$ be a right $A$-comodule $V$, with structure map $\tau:V\to V\otimes A$. The elements $f_{ij}\in A$, defined
by the equations $\tau(v_i)=\sum_{j\in J}v_j\otimes f_{ji}$ (for $i\in I$) are called the coefficient elements of
$A$ (relative to this basis). We write $\cf(V)$ for the subspace of $A$ spanned by all elements $f_{ij}$, $i,j\in
I$. It is independent of the choice of basis.  

\q We now take $S=A^*$,  the dual algebra of $A$. We identify $A$ with a subspace of $S^*$ via the natural map
$\eta:A\to A^*$,
$\eta(a)(s)=s(a)$, for $a\in A$, $s\in S$. We write $\Comod(A)$ for the category of right $A$-comodules and
$\comod(A)$ for the category of finite dimensional right $A$-comodules. 

\q A locally finite left $S$-module $V$ will be called. $A$-{\it admissible} if $\cf(V)\subseteq  A$. Given $V\in \Comod(A)$,
with structure map $\tau:V\to V\otimes A$, we regard $V$ as a left $S$-module via the product $sv=(1\otimes
s)\tau(v)$,
$v\in V$, $s\in S$. In this way we obtain an equivalence of categories between the category of right $A$-comodules
and $A$-admissible left $S$-modules.

\q Let $R$ be a subalgebra of $S$. We define 
$$C(A,R)=\{f\in A\vert f(rs)=f(sr)\hbox{ for all } r\in R, s\in S\}.$$
Viewing $A$ as an $S$-bimodule, and hence $R$-bimodule, we have 
$C(R,S)=H^0(R,A)$. We write simply $C(A)$
for $C(S,A)$. Note that $C(A)$ may be described directly in terms of the coalgebra structure as \\ 
$\{f\in A\vert
\de(f)=T\circ \de(f)\}$, where $T:A\otimes A\to A\otimes A$ is the twist map, i.e., the linear map taking $f\otimes
g$ to $g\otimes f$, for $f,g\in A$.

\q Let $V$ be a finite dimensional right $A$-comodule which, viewed as a left $S$-module, affords the representation
$\rho:S\to\End_k(V)$. Then, for $\theta\in\End_R(V)$, the element
$\chi_\theta$ of $S^*$, defined by $\chi_\theta(s)=\trace(\rho(s)\theta)$ belongs to $A$. If 
$v_1,\ldots,v_n$ is a $k$-basis of $V$ with coefficient functions $f_{ij}$ and if $\theta:V\to V$ is given by
$\theta(v_i)=\sum_{j=1}^n \theta_{ji}v_j$ then $\chi_\theta=\sum_{i,j=1}^n \theta_{ij}f_{ji}$. Moreover, from the familiar
property of the trace function, we have $\chi_\theta\in C(A,R)$. 

\q We define $C_0(R,A)$ to the subspace 
$\{\chi_\theta\vert \theta\in\End_R(V)\}$ of $C(A,R)$ and put  $C_0(R,A)=\sum_V C_0(R,A;V)$, where the sum is over all
$V\in\comod(A)$.

\q For $V\in\comod(A)$ the element $\chi_V$ of $S^*$ defined by $\chi_V(s)=\trace(s,V)$ belongs to $A$. In fact if
$V$ has basis $v_1,\ldots,v_n$ and corresponding coefficient elements $f_{ij}$, $1\leq i,j\leq n$, then we have
$\chi_V=f_{11}+\cdots+f_{nn}$. By arguments similar those of Section 2 we have that $C_0(A)$ is spanned by all
$\chi_V$, $V\in\comod(A)$. We write $K_0(\comod(A))$   for the Grothendieck group of
$\comod(A)$ and, for $V\in\comod(A)$ denote the class of $V$ in $K_0(\comod(A))$ by $[V]$.

\q We have the following result analogous to Section 2, (2).

\begin{proposition}  (i) The linear map $k\otimes_\zed K_0(\comod(A))\to C_0(A)$, taking $1\otimes [V]\to\chi_V$ for
$V\in\comod(A)$, is an isomorphism, equivalently if $V_\lambda$, $\lambda\in\Lambda$,  is a complete set of
pairwise non-isomorphic irreducible left comodules then $C(A)$ has $k$-basis
$\phi_\lambda$, $\lambda\in\Lambda$, where $\phi_\lambda=\chi_{V_\lambda}$, $\lambda\in\Lambda$.

(ii) If $A$ is a bialgebra then the above map is a $k$-algebra homomorphism.
\end{proposition}

\q We shall say that the coalgebra $A$ is constrained if $C_0(A)=C(A)$. If $A$ is finite dimensional we say that $A$
is a quasi-hereditary coalgebra if the dual algebra $A^*$ is quasi-hereditary.  We say that $A$ is {\it locally
quasi-hereditary} if each finite dimensional subcoalgebra is contained in a quasi-hereditary subcoalgebra.

\begin{proposition}  If $A$  is locally quasi-hereditary then it is constrained.
\end{proposition}

\begin{proof}  Suppose first that $A$ is finite dimensional and the dual algebra  $A^*$ quasi-hereditary. We identify $A$ with the double dual $(A^*)^*$ in the usual way. 
Then $C(A)=C(A^*)=C_0(A^*)=C_0(A)$ since $A^*$ is quasi-hereditary.  In general we consider $f\in C(A)$. Then $f\in C(B)$ for some finite dimensional quasi-hereditary 
subalgebra $B$ of $A$ and $C(B)=C_0(B)\subseteq C_0(A)$ so $f\in C_0(A)$. Hence $C(A)=C_0(A)$ and $A$ is constrained.
\end{proof}

\bs\bs\bs\bs

%-------------------------------------------------

%section6.tex

\section{Some remarks on relative invariants for quantum linear groups}
\bs\rm

\q  In this section we consider the algebra of invariants of the coordinate algebra of a quantum group $G$ under the action by conjugation of a (quantum) subgroup $H$. 
We are especially interested in the case in which $G$ is a product of general linear groups.  In the final section we consider the case $H=G$ and give more explicit results.

\q  Let $k$ be a field. By the statement \lq\lq$G$ is a quantum group" (over $k$) we  indicate that we
have in mind a Hopf algebra over $k$, denoted  $k[G]$ and called the coordinate algebra of $G$. By the statement \lq\lq
$\phi:G\to H$ is a morphism of quantum groups" we  indicate that $G$ and $H$ are quantum groups and we have in
mind a morphism of Hopf algebras $k[H]\to k[G]$, called the comorphism of $\phi$ and denoted $\phi^\sharp$. We say
that $H$ is a (quantum) subgroup of the quantum group $G$ to indicate that we have in mind a Hopf ideal $I_H$, called
the defining ideal of $H$ and that $H$ is the quantum group with coordinate algebra $k[G]/I_H$. The inclusion map
$i:H\to G$ is the morphism of quantum groups whose comorphism $i^\sharp$ is the natural map $k[G]\to k[G]/I_H$.

\q Let $G$ be a quantum group over $k$. By a left $G$-module we mean a right $k[G]$-comodule. We write $\mod(G)$
for the category of finite dimensional  left $G$-modules (i.e. the category of finite dimensional  right $k[G]$-comodules).

\q  Suppose now   that $A=(A,\de,\ep)$ is a coalgebra. We extend the definition of $C(A)$ to bicomodules.   By
an $A$-bicomodule we mean a triple
$(U,\lambda,\rho)$, where $U$ is a $k$-space and $\lambda:U\to A\otimes U$ and $\rho:U\to U\otimes A$ are linear maps
such that $(U,\lambda)$ is a left $A$-comodule, $(U,\rho)$ is a right $A$-comodule and 
$$(1\otimes \rho)\circ\lambda
=(\lambda\otimes 1)\circ \rho:U\to A\otimes U\otimes A.$$

\q Let $\Gamma=A^*$ be the dual algebra. Then an $A$-bicomodule $U$ is a $\Gamma$-bimodule with actions $\gamma\cdot u= (1\otimes \gamma) \rho(u)$ and $u\cdot \gamma= (\gamma\otimes 1) \lambda(u)$, for $\gamma\in \Gamma$, $u\in U$.

\q  We say that an element $u$ of an $A$-bicomodule $U$ is  stable if
$T\lambda(u)=\rho(u)$. Here $T:A\otimes U\to U\otimes A$ is the twist map, i.e. the $k$-linear map such that
$T(x\otimes y)=y\otimes x$, for $x\in A, y\in U$. We write $C(U)$ for the space of all stable elements of $U$. 

\q  It is easy to check that $C(U)$ is the zeroth Hochschild cohomology space: 
$$C(U)=H^0(\Gamma,U).\eqno{(1)}$$

\q Note
that $A=(A,\delta,\delta)$ is an $A$-bicomodule and that the more general definition of $C(U)$ agrees with the
earlier definition of $C(A)$, in case $U=A$.

 \q By a $G$-bimodule we mean a $k[G]$-bicomodule.    We now take $A=k[G]$ and let $S:k[G]\to k[G]$ be the antipode. Let $(U,\lambda,\rho)$ be an $G$-bimodule.  We write $C_G(U)$ for $C(U)$ is we wish to emphasise the role of $G$.  We write $U_\ad$ for the $k$-space $U$ regarded as a left $G$-module via the adjoint action, i.e., via the structure map 
 $\kappa:U\to U\otimes A$ given by
$$\kappa(u)=\sum_i u_i\otimes f_i''S(f_i')$$
 for $u\in U$ with
$$(1\otimes\rho)\lambda(u)=(\lambda\otimes 1)\rho(u)=\sum_i f_i'\otimes u_i\otimes f_i''.$$

\q If $V$ is a left $G$-module with structure map $\tau: V \to V\otimes k[G]$ then we write $V^G$ for the space of fixed points, i.e.,
$$V^G=\{v\in V \vert \tau(v)=v\otimes 1 \}.$$

\q  For $V,W$ be finite dimensional $G$-modules we regard $\Hom_k(V,W)$ as a $G$-bimodule  and identify 
 $V^*\otimes W$ with $\Hom_k(V,W)$ in the natural way.

\begin{lemma} Let $U$ be a $G$-bimodule and let $V,W$ be finite dimensional $G$-modules. 

(i) We have $C_G(U)\subseteq  U_\ad^G$ with equality if the antipode   $S:k[G]\to k[G]$   is injective.

(ii) We have $\Hom_G(V,W)\subseteq (V^*\otimes W)^G$, with equality if the antipode $S:k[G]\to k[G]$   is injective.
\end{lemma}

\begin{proof}  (i) Since a bicomodule is the union of its finite dimensional subbicomodules, we may assume $U$ finite
dimensional. Let $u_1,u_2,\ldots,u_m$ be a $k$-basis of $U$ and write 
$$\lambda(u_i)=\sum_j f_{ij}\otimes u_j \hbox{ and } \rho(u_i)=\sum_j u_j\otimes g_{ji}$$ for $1\leq i\leq m$, where
$\lambda:U\to A\otimes U$, $\rho:U\to U\otimes A$ are the bicomodule structure maps.

\q If $\sum_{i=1}^m t_iu_i\in C_G(U)$, for scalars $t_1,\ldots,t_m\in k$,  the condition $T(\lambda(\sum_i t_iu_i))=\rho(\sum_i t_iu_i)$ gives 
$$\sum_i t_jf_{ij}=\sum_i t_ig_{ji}$$ for all $1\leq j\leq m$. Hence we have
\begin{align*}
 \kappa(\sum_i t_iu_i&)=\sum_{i,j,r} t_iu_r\otimes g_{rj}S(f_{ij}) =\sum_{i,j,r}t_iu_r\otimes g_{rj}S(g_{ji})\\
 &=\sum_{i,r} t_iu_r\otimes \de_{ri}1=\sum_i t_iu_i\otimes 1
\end{align*} 
 and so $\sum_i t_iu_i\in U_\ad^G$. Hence we get $C_G(U)\leq U_\ad^G$.

\q Suppose now that $S$ is injective. If $\sum_i t_iu_i\in U_\ad^G$ then 
$$\sum_{i,j,r} t_iu_j\otimes g_{jr}S(f_{ir})=\sum_i t_iu_i\otimes 1$$ and hence
$$\sum_{i,r}t_ig_{jr}S(f_{ir})=t_j1$$ for all $1\leq i\leq m$. Let $1\leq s\leq m$. Multiplying by $S(g_{sj})$ and
summing over $j$  we get
$$\sum_{i,j,r}t_iS(g_{sj})g_{jr}S(f_{ir})=\sum_j t_jS(g_{sj}).$$ Hence we have $\sum_i t_iS(f_{is})=\sum_i
t_jS(g_{sj})$ so that $S(\sum_i t_i f_{is}-\sum_i t_ig_{si})=0$ and 
 $\sum_i t_i f_{is}=\sum_i t_ig_{si}$, for all $1\leq s\leq m$. This is the condition for $\sum_i t_iu_i$ to belong
to $C_G(U)$. Hence we have $U_\ad^G \leq C_G(U)$ and  $U_\ad^G=C_G(U)$.

(ii) We take $U=\Hom_k(V,W)$. 
\end{proof}

 \q  We  are  mainly be interested in the quantum general linear group defined by R. Dipper and the author, \cite{DD},   and variations of it.    We briefly recall its construction. Let $n$ be a positive integer. We write $A(n)$ for the $k$-algebra defined  by generators $c_{ij}$, $1\leq i,j\leq n$,  and relations: 
  \begin{align*}c_{ir}c_{is}=c_{is}c_{ir}&\phantom{pha} \hbox{ for all  $1\leq i,r,s\leq n$}\\
  c_{jr}c_{is}=qc_{is}c_{jr}&\phantom{pha} \hbox{ for all  $1\leq i<j\leq n, 1\leq r\leq s\leq n$}\\
  c_{js}c_{ir}=c_{ir}c_{js}+(q-1)c_{is}c_{jr}&\phantom{pha} \hbox{ for all  $1\leq i<j\leq n, 1\leq r< s\leq n$}.
 \end{align*}
 
Then $A(n)$ is a bialgebra with comultiplication $\de:A(n)\to A(n)\otimes A(n)$  and augmentation 
$\ep:A(n)\to k$ satisfying  $\de(c_{ij})=\sum_{r=1}^n c_{ir}\otimes c_{rj}$ and $\de(c_{ij})=\de_{ij}$, for
$1\leq i,j\leq n$. 

\q The determinant $d\in A(n)$ is defined by
$d=\sum_{\pi}\sgn(\pi)c_{1,1\pi}c_{2,2\pi}\ldots c_{n,n\pi}$ where $\pi$ runs over all permutations of
$\{1,2,\ldots,n\}$ and where $\sgn(\pi)$ denotes the sign of 
$\pi$.  

\q Now assume  $q\neq 0$. The bialgebra structure on $A(n)$ extends to the localization 
$A(n)_d$ of $A(n)$ at $d$ (with $\de(d^{-1})=d^{-1}\otimes d^{-1}$ and $\ep(d^{-1})=1$). Furthermore,  $A(n)_d$ is a
Hopf algebra. We write $k[G(n)]$ for $A(n)_d$ and call $G(n)$ the quantum general linear group of degree $n$. We
denote the antipode of $k[G(n)]$ by $S$. We have $S^2(f)=dfd^{-1}$, for $f\in k[G(n)]$, by \cite[4.3.15 Theorem]{DD},  in particular, $S:k[G(n)]\to
k[G(n)]$ is injective.

\q Let $T(n)$ be the subgroup of $G(n)$ whose defining ideal is generated by $\{c_{ij}\vert 1\leq i,j\leq n, i\neq
j\}$. Then $k[T(n)]$ is the  Laurent polynomial algebra
$k[t_1,t_1^{-1},\ldots,t_n,t_n^{-1}]$, where $t_i=c_{ii}+I_{T(n)}$, for $1\leq i\leq n$.

\q Let $X(n)=\zed^n$.  For each $\lambda=(\lambda_1,\ldots,\lambda_n)\in X(n)$  we have the one dimensional $T(n)$-module $k_\lambda$,
with structure map taking $x\in k_\lambda$ to $x\otimes t^\lambda$, where
$t^\lambda=t_1^{\lambda_1}t_2^{\lambda_2}\ldots t_n^{\lambda_n}$.  The modules $k_\lambda$, $\lambda\in X(n)$,
form a complete set of pairwise nonisomorphic simple $T(n)$-modules and a $T(n)$-module $V$ has a module decomposition
$V=\bigoplus_{\lambda\in X(n)}V^\lambda$, where $V^\lambda$ is a sum of copies of $k_\lambda$. We say that
$\mu\in  X(n)$ is a weight of $V$ if $V^\mu\neq 0$.

\q We write $X^+(n)$ for the set of dominant weights, i.e., the set of $\lambda=(\lambda_1,\ldots,\lambda_n)\in X(n)$ such that $\lambda_1\geq \lambda_2\geq \cdots\geq \lambda_n$.   We write $\Lambda^+(n)$ for the set of polynomial dominant weights, i.e., the set of $\lambda=(\lambda_1,\ldots,\lambda_n)\in X(n)$ such that $\lambda_1\geq \cdots \geq \lambda_n\geq 0$. 
We have the usual dominance partial order on $X(n)$. Thus, for  $\lambda=(\lambda_1,\ldots,\lambda_n), \mu= (\mu_1,\ldots,\mu_n) \in X(n)$ we write $\lambda\leq \mu$ if $\lambda_1+\cdots+\lambda_i\leq \mu_1+\cdots + \mu_i$ for $1\leq i <n$ and $\lambda_1+\cdots+\lambda_n =  \mu_1+\cdots + \mu_n$.

 \q For each $\lambda\in X^+(n)$ we have the induced module $\nabla(\lambda)$ and the Weyl module $\Delta(\lambda)$, as in \cite{DonkinX}.  For $\lambda\in X^+(n)$ the module $\nabla(\lambda)$ has simple socle $L(\lambda)$.  The modules $L(\lambda)$, $\lambda\in X^+(n)$, form a complete set of pairwise non-isomorphic simple $G(n)$-modules.  For $\lambda\in X^+(n)$ we write $M(\lambda)$ for the corresponding tilting module.  The  modules $\nabla(\lambda),\Delta(\lambda), L(\lambda), M(\lambda)$ have unique highest weight $\lambda$ and this occurs with multiplicity one. 
 
 \q For $\lambda=(\lambda_1,\ldots,\lambda_n)\in X^+(n)$ we set $\lambda^*=(-\lambda_n,-\lambda_{n-1},\ldots,-\lambda_1)$.  For $V\in \mod(G(n))$ we write $V^*$ for the dual module. For $\lambda\in X^+(n)$ the module $\nabla(\lambda)^*$ is isomorphic to $\Delta(\lambda^*)$, the module $\Delta(\lambda)^*$ is isomorphic to $\nabla(\lambda^*)$ and the module $L(\lambda)^*$ is isomorphic to $L(\lambda^*)$.

 \q We shall generalise this set-up to a direct product of quantum general linear groups.    Let $G_1,\ldots,G_m$ be quantum groups over $k$. Then we have the quantum group $G=G_1\times \cdots\times G_m$, whose coordinate algebra is the Hopf algebra $k[G_1]\otimes \cdots \otimes k[G_m]$.  If $V_i$ is a $G_i$-module, for $1\leq i\leq m$, then the tensor product $V_1\otimes \cdots \otimes V_m$ has  a natural $G$-module structure.  In particular, if $G$ is a quantum group then $G^m=G\times \cdots \times G$ ($m$ times) is a quantum group.  We have the diagonal embedding $\phi: G \to G^m$, whose comorphism $\phi^\sharp: k[G]\otimes \cdots \otimes  k[G]\to k[G]$ is multiplication.  Thus we have the quantum subgroup $H$ of $G^m$ whose defining ideal $I_H$ is $\Ker(\phi^\sharp)$.  We call this the diagonal subgroup.

\q Now let $a=(a_1,\ldots,a_m)$ be a  sequence of positive integers.  Then we have the group $G(a)=G(a_1)\times \dots \times G(a_m)$, as in \cite{DonkinX} and the subgroup $T(a)=T(a_1)\times \cdots \times T(a_m)$. We set $X(a) = X(a_1)\times \cdots \times X(a_m)$. For each $\lambda=(\lambda(1),\ldots,\lambda(m))\in X(a)$ we have the one dimensional $T(a)$-module $k_\lambda$, with structure map taking $x\in k_\lambda$ to $x\otimes t^{\lambda(1)}\otimes \cdots \otimes t^{\lambda(m)}$.  The modules $k_\lambda$, $\lambda\in X(a)$, form a complete set of pairwise non-isomorphic simple $T(a)$-modules and a $T(a)$-module $V$  has a  module decomposition $V=\bigoplus_{\lambda\in X(a)} V^\lambda$, where $V^\lambda$ is a sum of copies of $k_\lambda$. We say that $\mu\in X(n)$ is a weight of $V$ if $V^\mu\neq 0$.

\q We set $X^+(a)=X^+(a_1)\times \cdots \times X^+(a_m)$ and  $\Lambda^+(a)=\Lambda^+(a_1)\times \cdots \times \Lambda^+(a_m)$. The dominance partial order is defined on $X(a)$ by \\
$\lambda=(\lambda(1),\ldots,\lambda(m))\leq \mu=(\mu(1),\ldots,\mu(m))$ if $\lambda(i)\leq \mu(i)$, for all $1\leq i\leq m$. 

\q For $\lambda=(\lambda(1),\ldots,\lambda(m))\in X(a)$ (with $\lambda(i)\in X^+(a_i)$, $1\leq i\leq m$) we have the induced module  $\nabla(\lambda)=\nabla(\lambda(1))\otimes \cdots \otimes \nabla(\lambda(m))$, the Weyl module  $\Delta(\lambda)=\Delta(\lambda(1)\otimes \cdots \otimes \Delta(\lambda(m))$ and the simple $G(a)$-module $L(\lambda)=L(\lambda(1))\otimes \cdots \otimes L(\lambda(m))$.    For $\lambda\in X^+(a)$ we write $M(\lambda)$ for the corresponding tilting module. The  modules $\nabla(\lambda),\Delta(\lambda), L(\lambda), M(\lambda)$ have unique highest weight $\lambda$ and this occurs with multiplicity one.

\q For $\lambda=(\lambda(1),\ldots,\lambda(m))\in X^+(a)$ we define $\lambda^*=(\lambda(1)^*,\ldots,\lambda(m)^*)$. For $V\in \mod(G(a))$ we write $V^*$ for the dual module.   For $\lambda\in X^+(a)$ the module $\nabla(\lambda)^*$ is isomorphic to $\Delta(\lambda^*)$, the module $\Delta(\lambda)^*$ is isomorphic to $\nabla(\lambda^*)$ and the module $L(\lambda)^*$ is isomorphic to $L(\lambda^*)$.

\q By a good filtration of a finite dimensional $G(a)$-module $V$ we mean a filtration $0=V_0 \subseteq   V_1 \subseteq  \cdots  \subseteq  V_t=V$ such that  for each $1\leq i\leq n$ the section   $V_i/V_{i-1}$ is either $0$ or isomorphic to $\nabla(\lambda_i)$, for some $\lambda^i \in X^+(a)$.  For $\lambda\in X^+(a)$, the number of $i$ such that $1\leq i\leq t$ and $V_i/V_{i-1}$ is isomorphic to $\nabla(\lambda)$ is independent of the choice of good filtration, and will be denoted $(V:\nabla(\lambda))$.

\begin{remark} Since the antipode $S$ of $k[G(n)]$ satisfies $S^2(f)=dfd^{-1}$, for $f\in k[G(n)]$, it is a linear isomorphism. It follows that, for a finite string of positive integers,  the antipode of $G(a)$,  and hence of any quantum subgroup, is an isomorphism. In particular it is injective.
\end{remark}

\begin{definition}  A (quantum) subgroup $H$ of $G(a)$ will be called saturated if $\Hom_H(\Delta(\lambda),\blank)$ is exact on short exact sequences of $G(a)$-modules with a good filtration. 
\end{definition}

\begin{remark} Let $H$ be a quantum subgroup of $G(a)$.  Then $H$ is saturated if and only \, $(\blank)^H$ is exact on short exact sequences of $G(a)$-modules with a good filtration. This may be seen as follows. For a $G(a)$-module $V$ we have $\Hom_H(k,V)\iso V^H$ so that if $H$ is saturated then \, $(\blank)^H$ is exact on short exact sequences of $G(a)$-modules with a good filtration. 

\q Now suppose that  \, $(\blank)^H$ is exact on short exact sequences of $G(a)$-modules with a good filtration. Let $\lambda\in  X^+(a)$. If $M\in \mod(G(a))$ has a good filtration then $\nabla(\lambda^*)\otimes M$ has a good filtration, by \cite[Section 4, (3)(i)]{DonkinX}. Thus if $0\to V'\to V\to V''\to 0$ is a short exact sequence of $G(a)$-modules with a good filtration then so is 
$$0\to \nabla(\lambda^*)\otimes  V'\to \nabla(\lambda^*) \otimes V\to \nabla(\lambda^*)\otimes V''\to 0.$$
  Hence we get that 
  $$0\to (\nabla(\lambda^*)\otimes  V')^H \to (\nabla(\lambda^*) \otimes V)^H \to (\nabla(\lambda^*)\otimes V'')^H\to 0$$
  is exact and hence, by Lemma 6.1(ii), so is 
  $$0\to \Hom_H(\Delta(\lambda),V')\to \Hom_H(\Delta(\lambda),V)\to \Hom_H(\Delta(\lambda),V'')\to 0.$$
   Hence $H$ is saturated. 
\end{remark}

\begin{lemma}  Let $m\geq 1$. Then $G(n)$ regarded as a subgroup of $G(n)^m= G(n)\times \cdots \times G(n)$, via the diagonal embedding, is saturated. 
\end{lemma}

\begin{proof} Let $0\to V' \to V \to  V''\to 0$ be a short exact sequence of $G(n)^m$-modules with a good filtration.  Then $V'$, regarded as a $G(n)$-module, also has a good filtration, by \cite[Section 4, 3 (i)]{DonkinX}, hence $H^1(G(n),V')=0$, by \cite[Section 4,(2)]{DonkinX},  and so the sequence $0\to (V')^{G(n)}\to V^{G(n)}\to (V'')^{G(n)} \to 0$ is exact. Hence $G(n)$ is saturated.
\end{proof}

\q Let $\pi$ be a subset of $X^+(a)$. We say that a $G(a)$-module $V$ belongs to $\pi$ if each composition factor of $V$ has the form $L(\lambda)$, for some $\lambda\in \pi$.  Among all submodules belonging to $\pi$ of and arbitrary $G(a)$-module $V$ there is a unique maximal one and we denote this $O_\pi(V)$. Applying this to the left regular $G$-module $k[G]$ we obtain $O_\pi(k[G(a)])$. This is a subcoalgebra of $k[G(a)]$.

\q A subset $\pi$ of $X^+(a)$ is called saturated if whenever $\lambda, \mu\in X^+(a)$, $\lambda\in \pi$ and $\mu\leq \lambda$ then $\mu\in \pi$.  Suppose that $\pi$ is a finite saturated set. Then $A=O_\pi(k[G])$ is a finite dimensional subcoalgebra. We denote the dual algebra by  $S(a,\pi)$.  As in \cite[Section 4, (6)]{DonkinX} we have the following.

\begin{proposition} For a finite saturated subset $\pi$ of $X^+(a)$ the algebra $S(a,\pi)$ has complete set of pairwise non-isomorphic modules $L(\lambda)$, $\lambda\in \pi$.  Moreover the algebra $S(a,\pi)$ is quasi-hereditary, with respect to the dominance partial order on $\pi$ and for $\lambda\in \pi$, the module $\Delta(\lambda)$ is the corresponding standard module and $\nabla(\lambda)$ is the corresponding co-standard module.
\end{proposition}

This, together with Propositions 5.1 and 5.2,  gives the following.

\begin{proposition} Let $\pi$ be a saturated subset of $X^+(a)$. 

(i) The coalgebra $O_\pi(k[G(a))]$ is locally quasi-hereditary:

(ii)  Let $\phi_\lambda$, be the character of $L(\lambda)$, for $\lambda\in \pi$. Then the elements $\phi_\lambda$, $\lambda\in \pi$, form a basis for $C(O_\pi(k[G(a)])$. 
\end{proposition}

\q Let $G$ be a quantum group over $k$. We write $\Gamma_G$ for the dual algebra $k[G]^*$ of $k[G]$. Let   $A$ be a finite dimensional subcoalgebra of $k[G]$ and let $S=A^*$, the dual algebra.  The inclusion of $A$ in $k[G]$ gives rise to an algebra epimorphism $\Gamma_G\to S=A^*$. 

\q Let $H$ be a quantum subgroup of $G$. he natural map $k[G]\to k[H]$ gives rise to a $k$-algebra map $\Gamma_H\to \Gamma_G$.  We hence have  the composite $\theta:\Gamma_H\to S$. We write $R$ for the image of $\theta$.  

\q We say that a finite dimensional $G$-module $V$ is $A$-admissible if $\cf(V)\subseteq A$.   A subspace $U$ of an $A$-admissible $G$-modules is an $H$-submodule $V$  if and only if it is an $R$-submodule. Moreover, if $V'$ is also an $A$-admissible $G$-module then we have $\Hom_H(V,V')=\Hom_R(V,V')$. 

\q Now suppose that $G=G(a)$, for a finite sequence of positive integers $a$ and that $H$ is a saturated subgroup. Suppose $\pi$ is a finite saturated subset of $X^+(a)$ and that $A=O_\pi(k[G])$.  If $\lambda\in \pi$ and $V$ is an $A$-admissible $G$-module then we have 
$\Hom_R(\Delta(\lambda), V)=\Hom_H(\Delta(\lambda),V)$. It follows that $R$ is a saturated subalgebra of $S$, as in Section 4.  Hence we have $C_0(S,R)=C(S,R)$. Identifying $A$ with the dual of $S^*$ and using the fact that the action of $\Gamma_H$ on $A$ factors though $R$ we get
$$C_0(S,R)=C(S,R)=H^0(R,A)=H^0(\Gamma_H,A)=C_H(A)=A^H.$$

\begin{proposition} Let $a$ be a finite sequence of positive integers.
Suppose that  $H$ is a saturated subgroup of $G(a)$.   Let $\pi$ be a saturated subset of $X^+(a)$. 
Then the algebra of relative class functions $A(\pi)^H$ is spanned by the shifted trace functions $\chi_{\theta,M(\lambda)}$,  with $\lambda\in \pi$, $\theta\in \End_H(V)$.
\end{proposition}

\begin{proof} By a local finiteness argument we are reduced to the case in which $\pi$ is finite, considered above.
\end{proof}

\begin{remark} This applies in particular to the case in which $G(a)=G(n)^m$ and $H=G(n)$,the diagonal subgroup. Taking $\pi=\Lambda^+(n)^m$, we get that $A(\pi)$ is  the $m$-fold tensor product $A(n)\otimes\cdots \otimes  A(n)$, as in the proof of  \cite[Section 2, Theorem]{DonkinQ}.   The algebra of invariants $(A((n)\otimes\cdots \otimes  A(n))^{G(n)}$
 is the $q$-analogue of the algebra of polynomial invariants for action of the general linear group on $m$-tuples of matrices, considered in \cite{DonkinQ}, where an explicit set of generators is given.  However, there is more work to be done to give an analogous  set of generators in the quantum case.  Another solution  to this problem of describing generators  (in the classical case) was given by Domokos and Zubkov, \cite{DZ}.
\end{remark}

\begin{remark}  One  may replace the modules $M(\lambda)$ in the proposition by any collection of modules which contain all modules $M(\lambda)$, $\lambda\in \pi$, as summands. For $r\geq 0$ we write $\Lambda^+(n,r)$ for the subset of $\Lambda^+(n)$ consisting of the partitions of $r$ with at most $n$ parts. For a partition $\lambda$ we write $\lambda'$ for the transpose partition. We have the natural $G(n)$-module $E$, as in \cite{DonkinAA}, and, for $r\geq 0$, the exterior power $\tbw^r E$.  Thus for a finite string of non-negative integers  $\alpha=(\alpha_1,\ldots,\alpha_h)$ the module $\tbw^\alpha E= \tbw^{\alpha_1}E\otimes \cdots \otimes \tbw^{\alpha_h} E$.  For $\lambda\in \Lambda^+(n,r)$ the module $\tbw^{\lambda'} E$ is a tilting module for $G(n)$  in which $M(\lambda)$ occurs as component.  More generally for a string of non-zero integers $a=(a_1,\ldots,a_m)$ and $\lambda=(\lambda(1),\cdots,\lambda(m))\in \Lambda^+(a)$ the tilting module $M(\lambda)=M(\lambda(1))\otimes \cdots \otimes M(\lambda(m))$  for $G(a)$,  occurs as a component of $\tbw^{\lambda(1)'} E\otimes \cdots \otimes \tbw^{\lambda(m)'}E$.  Thus $(A(n)^{\otimes m})^{G(n)}$ is spanned by the elements $\chi_{\theta,V}$, where $V$ is a module of the form  $\tbw^{\lambda(1)'} E\otimes \cdots \otimes \tbw^{\lambda(m)'}E$, with $\lambda(1),\ldots,\lambda(m)\in \Lambda^+(n)$ and $\theta\in \End_{G(n)}(V)$.  It is modules of this form that are used in the application of Proposition 6.8 in the classical case, \cite{DonkinQ}.
\end{remark}

\bs\bs\bs\bs

%---------------------------------------------------

%section7.tex

\section{The coordinate algebra of a general linear group as a module over its algebra of class functions}
\bs\rm

\q We here concentrate on the \lq\lq absolute" case $G=H=G(n)$.

\begin{proposition} (i)  $C(A(n))$ is the free polynomial algebra on $\phi_1,\ldots,\phi_n$, where $\phi_r$ is the natural character
of $\tbw^r E$, $1\leq r\leq n$.

(ii) $C(k[G(n)])$ is the localization of $C(A(n))$ at the determinant $d=\phi_n$.
\end{proposition}

\begin{proof} (i) Let $\psi_\lambda$ be the character of $L(\lambda)$, for $\lambda\in \Lambda^+(n)$. For $\lambda\in \Lambda^+(n)$ the module $\tbw^{\lambda'} E$ has unique highest weight $\lambda$.  Hence, for $r\geq 1$, the characters of the modules $\tbw^{\lambda'} E$, $\lambda\in \Lambda^+(n,r)$, and the characters of the modules $L(\lambda)$, $\lambda\in \Lambda^+(n,r)$ are related by a unitriangular matrix.  Hence, from Proposition 6.7, the characters of  the $\tbw^{\lambda'} E$, $\lambda\in \Lambda^+(n)$, form a basis for $C(A(n))$. The result follows. 

(ii) Follows from (i). 
\end{proof}

\q For the case in which $k$ has characteristic $0$ and $q$ is not a root of unity, see \cite{DLA}.

\q   We now consider the coordinate algebra $k[G(n)]$ as a module over the algebra of invariants $C(k[G(n)])$ and the algebra $A(n)$ as a module over $C(A(n))$. 

\q We need to make a certain observation on flat modules. This is presumably well known but we include it since we do not know a precise reference (but cf \cite[4.E)]{Matsumura}).
 Let $A$ be a subring of a commutative ring $B$. Recall that if $B$ is flat over $A$ then 
$\Tor^B_i(B\otimes_A M,N)$ is isomorphic to $\Tor_i^A(M,N)$, for all $i\geq 0$,  all $A$-modules $M$ and all $B$-modules $N$, see \cite[Proposition 3.2.9]{Weibel}.

\bs

\begin{lemma}  Let $B$ be a commutative ring and $A$ a subring. Suppose $A$ and $B$ are regular Noetherian and that  $B$  is   flat  and finite  over  $A$.   A $B$-module  is flat if and only if it is flat as an $A$-module. 
\end{lemma}

\begin{proof} A flat $B$-module  is also flat as an  $A$-module, see \cite[(3.B)]{Matsumura}.

\q Now suppose $M$ is a $B$-module which is flat as an  $A$-module.  To show that $M$ is flat over $B$, it suffices, by \cite[Chapter II, Section 3.4 Corollary]{Bourbaki},  to show  that, for each maximal ideal $Q$ of $B$, the localization $M_Q$ is a flat as a module over the local ring $B_Q$.  Let $P=Q\cap A$.  The natural map $M_P\to M_Q$ is an isomorphism (since $B$ is integral over $A$) so that $M_Q$ is flat as an $A_P$-module. Thus we may replace $A$ by $A_P$ and $B$ by $B_Q$ and $M$ by $M_Q$. Thus we may assume that the  inclusion of $A$ in $B$ is a homomorphism of regular  local rings.  

\q Suppose, for a contradiction that $M$ is not flat.
Then we have \\
$\Tor_1^B(M,N)\neq 0$ for some finitely generated $B$-module $N$ (see \cite[Chapter 2, Section 3, Theorem 1]{Matsumura}).  Let $i$ be maximal such that there exists a finitely generated $B$-module $N$ with $\Tor_i^B(M,N)\neq 0$.  Since $N$ has a finite filtration with factors of the form $B/Q$, with $Q\in\Spec(B)$ (see e.g. \cite[(7.E)]{Matsumura}) we must have $\Tor_i^B(M,B/Q)\neq 0$, for some such $Q$.  Let $P=B\cap Q$, let $F$ denote the field of fractions of $A/P$ and let $K$ denote  the field of fractions of $B/Q$.  Then (by integrality) the natural map $B\otimes_A F\to K$ is an isomorphism.  Now the inverse   isomorphism $K\to B\otimes_A F$ takes $B/Q$ into $B\otimes_A X$, for some finitely generated $A$-submodule $X$ of $F$. However, we have 
$\Tor_j^B(M,B\otimes_A X)=\Tor_j^A(M,X)=0$, for all $j\geq 1$, since $M$ is a flat $A$-module. Thus a  short exact sequence $0\to B/Q\to B\otimes_A X\to Y\to 0$, of $B$-modules, gives rise to an isomorphism $\Tor^B_{i+1}(M,Y)\to\Tor_B^i(M,B/Q)$ and $\Tor^B_{i+1}(M,Y)$ is $0$ by assumption so that $\Tor_i^B(M,N)=0$, a contradiction, and we are done.
\end{proof}

\begin{lemma} $A(n)$ is a flat $C((A(n))$-mdoule. 
\end{lemma}

\begin{proof} The case $q=1$ is covered by \cite[2.1 Proposition]{DonkinO} so we assume   $q$ to be a primitive $l$th root of unity, for some $l>1$.    Let $\barA(n)$ be the free commutative  polynomial algebra 
$k[x_{11},x_{12},\ldots,x_{nn}]$.  Then $\barA(n)$ is naturally a bialgebra with diagonal $\barA(n)\to\barA(n)\otimes\barA(n)$ taking $x_{ij}$ to $\sum_{r=1}^nx_{ir}\otimes x_{rj}$ and augmentation $\barA(n)\to k$ taking $x_{ij}$ to $\de_{ij}$, for $1\leq i,j\leq n$. We write $\bard$ for $\det(x_{ij})$. Then the localization $\barA(n)_{\bard}$ is the coordinate algebra of the general linear group scheme $\barG(n)$.  Let $\barT(n)$ be the subgroup scheme with defining ideal generated by the $x_{ij}$, with $1\leq i,j\leq n$, $i\neq j$.

We have the Frobenius map $F:G(n)\to\barG(n)$, whose comorphism $\sharpF:k[\barG(n)]\to k[G(n)]$ takes $x_{ij}$ to $c_{ij}^l$, for $1\leq i,j\leq n$. The map $\sharpF$ restricts to an algebra monomorphism $\barA(n)\to A(n)$ and $A(n)$ is free over $\barA(n)$,  on basis $x_{11}^{a_{11}}x_{12}^{a_{12}}\cdots x_{nn}^{a_{nn}}$, with $0\leq a_{11},\ldots,a_{nn}<l$, see \cite[1.3.3  Theorem]{DD}.

\q Let $\barC=\barC(n)$ be the algebra of invariants of $\barA(n)$ for the conjugation action by $\barG(n)$. Let $\barE$ be the natural $\barG(n)$-module.  Then $\barC(n)$ is a free polynomial algebra on $\barphi_1,\ldots\barphi_n$, where $\bar\phi_i$ is the character of $\tbw^i\barE$, $1\leq i\leq n$, by  Proposition 7.2 (i). Moreover, by \cite[2.2 Theorem]{DonkinO}, $\barA(n)$ is free, in particular flat, as a $\barC(n)$-module.  Hence  (by the transitivity of flatness)  $A(n)$ is a flat $\barC(n)$-module. 

\q  We now  apply  Lemma 7.2  to deduce that $A(n)$ is a flat $C(n)$-module.   We need to check that $C(n)$ is  finite over  $\barC(n)$.  We consider the restriction map $C(n)\to k[T(n)]$. 
Let $t_i\in k[T(n)]$ be the restriction of $c_{ii}$ to $T(n)$. Then $k[T(n)]$ is the  Laurent  polynomial algebra   $k[t_1,t_1^{-1},\ldots,t_n,t_n^{-1}]$. The restriction map $C(n)\to k[T(n)]$ is injective and the image is the algebra of symmetric functions in $t_1,\ldots,t_n$, which we denote by $\Lambda$.  Similarly $k[\barT(n)]$ is the Laurent polynomial algebra   $k[\bart_1,\bart_1^{-1}.\ldots,\bart_n,\bart_n^{-1}]$, where $\bart_i$ is the restriction of $x_{ii}$, $1\leq i\leq n$,  and the restriction map $\barC(n)\to k[\barT(n)]$ is injective with image the algebra of symmetric functions in $\bart_1,\ldots,\bart_n$, which we denote by $\barLambda$. The Frobenius morphism $F:G(n)\to\barG(n)$ induces an embedding $F^\sharp:k[\barT(n)]\to k[T(n)]$, taking $\bart_i$ to $t_i^l$, $1\leq i\leq n$. We identify $k[\barT(n)]$ with a subalgebra of $k[T(n)]$ via this embedding. Now $k[t_1,\ldots,t_n]$ is finite over the $k[\bart_1,\ldots,\bart_n]$ and hence over the algebra of invariants (under the symmetric group of degree $n$) $\barLambda$.  Since $\barLambda\leq \Lambda\leq k[t_1,\ldots,t_n]$, we have that $\Lambda$ is finite over $\barLambda$.  Moreover, as is well known $k[t_1,\ldots,t_n]$ is free and hence flat over $\Lambda$. Moreover, $k[t_1,\ldots,t_n]$ is free over $k[\bart_1,\ldots,\bart_n]$ and  $k[\bart_1,\ldots,\bart_n]$ is free over $\barLambda$.  Hence $k[t_1,\ldots,t_n]$ is free, and hence flat over $\barLambda$. By integrality, we now have (see \cite[(4.A) and (5.E)]{Matsumura})  $k[t_1,\ldots,t_n]$ is faithfully flat over $\barLambda$ and $k[t_1,\ldots,t_n]$ is faithfully flat over $\Lambda$. However, it follows from the definitions that if we have rings $A\leq B\leq C$ with $C$ faithfully flat over $A$ and $B$ faithfully flat over $A$ then $C$ if faithfully flat over $B$. Thus $\Lambda$ is flat over $\barLambda$. The restriction map $C(n)\to \Lambda$ takes $\barC(n)$ to $\barLambda$. It therefore follows that $C(n)$ is flat and integral over $\barC(n)$. Since $A(n)$ is flat over $\barC(n)$,  it is also flat over $C(n)$.
\end{proof}

\q  We set $G=G(n)$, $T=T(n)$, $A=A(n)$ and $J=C(A(n))$. We shall consider $A$ as a module for  $G$ (via the adjoint action)  and $J$.  For a $J$-module $V$ and a $k$-space $M$ we write $|M|\otimes V$ for the vector space $M\otimes V$, viewed as a $J$-module via the action $c(m\otimes v)=m\otimes cv$ (for $c\in J$, $m\in M$, $v\in V$). By a $(J,G)$-module we mean a $k$-vector space $V$ which is a $J$-module and a $G$-module in such a way that the action of $J$ is by $G$-module endomorphisms of $V$.   Morphisms of $(J,G)$-modules, $(J,G)$-submodules etc are defined in the usual way. 
Given a $(J,G)$-module $V$ and a  $G$-module $M$ we regard the $J$-module $|M|\otimes V$ as a $(J,G)$-module, with $G$ acting diagonally. We sometimes regard $J$ as a $(J,G)$-module, with $J$ acting via the regular action and $G$ acting trivially.

\q With Lemma 7.3 in place the argument of \cite[2.2 Theorem]{DonkinO}  now goes through and we obtain the following.  (The second assertion follows from the first by localizing at the determinant.)

\begin{theorem} Suppose that $q$ is a root of unity. Let $\lambda(1),\lambda(2),\ldots$ be a labelling of the elements of $X^+(n)$ such that  $i<j$ whenever $\lambda(i) <\lambda(j)$. 
Then $A$ has a filtration $0=Y(0)\subseteq Y(1)\subseteq  \ldots$, as a $(J,G)$-module, with $Y(i)/Y(i-1)\iso |F(i)|\otimes J$,   where $F(i)$ is a direct sum of $\dim \nabla(\lambda(i))^T$ copies of $\nabla(\lambda(i))$, for $i\geq 1$.

\q Moreover $k[G(n))]$ has a  filtration $0=Z(0)\subseteq Z(1)\subseteq  \ldots$, as a \\
$(C(k[G]),G)$-module, with $Z(i)/Z(i-1)\iso |F(i)|\otimes C([k[G])$,    for $i\geq 1$.
\end{theorem}

\begin{corollary}  Suppose that $q$ is a root of unity.  Then $A$ is a free $J$-module.
\end{corollary}

\q We will give a  strong  version of Theorem 7.4 when $q$ is not a root of unity. Our argument involves the $A$-regularity of the sequence $\phi_1,\ldots,\phi_n$. 

\q    We have the coalgebra decomposition and algebra grading $A=\bigoplus_{r=0}^\infty A_r$, where $A_r=A(n,r)$ is the span of the elements $c_{i_1j_1}\ldots c_{i_rj_r}$,  with  \\
$1\leq i_1,j_1,\ldots, i_r,j_r\leq n$.  Note that $J$ is a homogenous subspace of $A$. We set $J^+=\bigoplus_{r> 0}J_r$.

\q We consider the following set-up. Let $R$ be a discrete valuation ring with field of fractions $K$ and maximal ideal $\m$ generated by an element $\pi$, and let   $k=R/\m$ be the residue field.   Let $q$ be an element of $R\backslash \m$ and $\barq$ be the image of $q$ in $k$.   We write $A_K$ for $A_{K,q}(n)$, write $A$ for the $R$-subalgebra generated by the coefficient elements $c_{ij}$, $1\leq i,j\leq n$. We write $\phi_1,\ldots, \phi_n$ for the characters of the exterior powers of the natural $G(n)$-module over $K$, as in Proposition 7.1. We write $\barA$ for the algebra $A_{k,\barq}(n)$. We have the natural map $A\to \barA$ and, for $a\in A$, we write $\bara$ for its image in $\barA$. 
Now $A$ inherits a grading $A=\bigoplus_{r\geq 0} A_r$, from $A_K=A_{K,q}(n)$. Also $\barA$ is graded and the natural map $A\to \barA$ takes $A_r$ to $\barA_r$, $r\geq 0$.

\begin{lemma} If $\barphi_1,\ldots,\barphi_n$ is an $\barA$-regular  sequence  then $\phi_1,\ldots,\phi_n$ is an $A$-regular sequence.
\end{lemma}

\begin{proof} Suppose $\phi_1 a=0$, for some $a\in A$. Then $\barphi_1 \bara=0$ and, since $\barphi$ is not a zero divisor, $\bara=0$, i.e., $a\in \pi A$. We write $a=\pi b$, for some $b\in A$. Then $\pi \phi_1 b=0$ so that $\phi_1 b=0$. Hence $b\in \pi A$ and $a\in \pi^2 A$. Inductively, we get $a\in \pi^s A$, for all $s\geq 1$. But $A$ is free over $R$ so that $\bigcap_{s\geq0}  \pi^s A=0$ and $a=0$. Thus $\phi_1$ is not a zero divisor.

\q Now suppose that $1\leq m<n$ and $a\in A$.  We set $Y=\phi_1 A+\cdots + \phi_m A$ and $X=\{a\in A\vert \phi_{m+1} a\in Y\}$.  Now $X$ and $Y$ inherit a gradings from $A$ and 
$$\bigcap_{s\geq 0} \pi^s (A/Y)=\bigoplus_{r= 0}^\infty \bigcap_{s\geq 0} \pi^s (A_r/Y_r)$$
moreover, for each $r$, the $R$-module $A_r/Y_r$ is finitely generated and hence (e.g. by Nakayama's Lemma) $\bigcap_{s\geq 0} \pi^s (A_r/Y_r)=0$. Hence we have \\
$\bigcap_{s\geq 0} \pi^s (A/Y)=0$, i.e.,
$$ \bigcap_{s\geq 0} (\pi^s A +Y)=Y.\eqno{(1)}$$

\q Let $x\in X$.  Thus we have $\phi_{m+1} x= \phi_1 x_1+\cdots +\phi_m x_m$, for some $x_1,\ldots,x_n\in A$ and hence 
$\barphi_{m+1} \barx=\barphi_1 \barx_1+\cdots \barphi_m \barx_m$. But  $\barphi_1,\ldots,\barphi_n$ is a regular sequence for $\barA$ so there exist elements $y_1,\ldots,y_m$ of $A$ such that $\barx=\barphi_1 \bary_1+\cdots +\barphi_m \bary_m$ and we deduce that for any $x\in X$ there exist $y_1,\ldots,y_m,z\in A$ such that 
$$x=\phi_1 y_1+\cdots+ \phi_m y_m + \pi z  \eqno{(2)}.$$
In particular $X\subseteq Y+ \pi A$.

\q Now assume that $s$ is a positive integer such that  $X\subseteq Y+ \pi^s A$.

\q We fix $a\in X$ and write 
$$a= \phi_1 b_1+\cdots + \phi_m b_m + \pi^s c \eqno{(3)}$$
 for elements $b_1,\ldots,b_m,c$ of $A$.
 
 \q Multiplying (3) by $\phi_{m+1}$ and using  that $\phi_{m+1}a\in Y$ we get $\phi_{m+1} \pi^sc \in Y$.    We claim that $\phi_{m+1}\pi^s c \in \pi \sum_{i=1}^m \phi_i A$. If not  we choose $1\leq j\leq m$ as small as possible such that 
 $$\phi_{m+1} \pi^s c \in \sum_{1\leq i\leq j} \phi_i A + \pi(\sum_{j+1\leq i\leq m} \phi_iA). \eqno{(4)}$$
   We write 
 $$\phi_{m+1} \pi^ s c= \sum_{1\leq i\leq j} \phi_i u_i + \sum_{j+1\leq i\leq m} \pi \phi_iv_i\eqno{(5)}$$
 for elements $u_i$, $1\leq i\leq j$ and $v_i$, $j+1\leq i\leq m$ of $A$.  Taking (5) modulo $\pi$ we get $\sum_{1\leq i\leq j} \barphi_i \baru_i=0$.  Since $\barphi_1,\ldots,\barphi_n$ is a regular sequence for $\barA$, we must have 
 $$u_j=\sum_{1\leq i<j} \phi_i e_i+\pi f$$
  for some elements $e_i$ (with $1\leq i<j$),  and $f$ of $A$.
  
  \q Now we have 
  \begin{align*}\phi_{m+1} \pi^s c&=\sum_{1\leq i<j} \phi_iu_i + \phi_j ( \sum_{1\leq i<j } \phi_ie_i +\pi f)+\sum_{j+1\leq i\leq n} \pi \phi_i v_i\cr
  &=\sum_{1\leq i<j} \phi_i(u_i+\phi_j e_i)+\pi \phi_j f+ \sum_{j+1\leq i\leq n} \pi \phi_i v_i
  \end{align*}
 contradicting the minimality of $j$. 
 
 \q Thus we have $\phi_{m+1} \pi^s c \in  \pi \sum_{i=1}^m \phi_i A$.   Hence we have $\phi_{m+1} \pi^{s-1} c\in  \sum_{i=1}^m \phi_i A$, i.e., $\pi^{s-1} c\in X$.  Now  we may write 
 $$\pi^{s-1} c=\phi_1g_1+\cdots \phi_m g_m +\pi^s h$$
 for some $g_1,\ldots,g_m, h\in A$.  Thus from (3) we have
  $$a=\phi_1 a_1+ \cdots + \phi_m a_m +\pi (\phi_1 g_1+ \cdots +\phi_m g_m + \pi^s h)$$
and then $a$ belongs to $\phi_1 A+ \cdots +\phi_m A + \pi^{s+1} A=Y+\pi^{s+1}A$. .
 
 \q  By induction we have that $a\in Y + \pi^s A$, for all $s$, i.e., $a\in \bigcap_{s=0}^\infty (Y+\pi^s A)$, which, by (1), is $Y$. Thus $a\in Y$, i.e., we have $a=\phi_1 w_1+\cdots +\phi_m w_m$, for some $w_1,\ldots,w_m\in A$. Thus $\phi_1,\ldots,\phi_n$ is an $A$-regular sequence.
 \end{proof}

\begin{theorem} For any field and  any $0\neq q\in k$, the characters $\phi_1,\ldots,\phi_n$ form an $A_{k,q}(n)$-regular sequence.
\end{theorem}

\begin{proof} We give the proof in several steps.

\medskip

\it Step 1.  \q\rm In the classical case $q=1$ the elements $\phi_1,\ldots,\phi_n$ form a regular sequence (and if $k$ is algebraically closed, the ideal generated  by these elements is the defining ideal of the nullcone in the variety of $n\times n$ matrices).   

\medskip

\it Step 2.  \q\rm Next assume $q\neq 1 $ and $q$ is a root of unity.   By Corollary 7.5, $A$ is free over $J$. Let $b_i$, $i\in I$, be a $J$-basis of $A$.   If $a\in A$ and $\phi_1 a=0$ then writing $a=\sum_{i\in I} f_i b_i$, with $f_i\in J$, we have $\sum_{i\in I} \phi_1 f_i b_i=0$. Hence for each $i$ we have $\phi_1 f_i=0$ and $J$ is a domain so that $f_i=0$.  Hence $a=0$ and $\phi_1$ is not a zero divisor. 

\q Now suppose that $1\leq m<n$. Suppose $a\in A$ and \\
$$\phi_{m+1} a=  \phi_1 a_1+\ldots + \phi_m a_m$$
with $a_1,\ldots, a_m\in A$.  We write $a=\sum_{\in I} g_i b_i$ and $a_r=\sum_{i\in I} g_{ir} b_i$, for $1\leq r\leq m$, with all $g_i, g_{ir}\in J$.  Thus, for each $i\in I$, we have
$$\phi_{m+1} g_i= \phi_1 g_{i1}+ \cdots + \phi_m g_{im}.$$
But $J$ is the (commutative)  polynomial algebra in $\phi_1,\ldots,\phi_n$. In particular $\phi_1,\ldots,\phi_n$ is a $J$-regular sequence. Hence each $g_i\in \phi_1 J + \phi_2 J +\cdots + \phi_m J$ and $a=\sum_{i\in I} g_i b_i\in \phi_1 A+\cdots + \phi_m A$.

\medskip

\it Step 3. \q\rm  By a similar argument  to that of Step 2 one sees, for general $0\neq q\in k$ and an extension field $K$ of $k$ the elements $\phi_1,\ldots,\phi_n\in A_{k,q}(n)$ form a regular sequence if and only if the elements $1\otimes \phi_1,\ldots,1\otimes \phi_n$ of $A_{K,q}(n)= K\otimes A_{k,q}(n)$ form a regular sequence. 

\medskip

\it Step 4. \q\rm  Now assume  that $q$ is algebraic over the prime field $k_0$   but not a root of unity.  By Step 3, we may replace $k$ by $k_0(q)$. If $k$ has positive characteristic this implies that $q$ is a root of unity. Hence  $k=\que(q)$ and that $q$ satisfies an equation 
$$r_0 q^d + r_1 q^{d-1} + \cdots + r_d=0$$
for some positive integer $d$ and integers $r_0,\ldots,r_d$ with $r_0\neq 0$. 

\q We choose a rational prime $p$ not dividing $r_0$. Let $R_0$ be the subring of $k$ consisting of the elements $u/v$, with $u,v$ integers and $p$ not dividing $v$. Let $R_1=R_0(q)$ and choose a maximal ideal $\m$ of $R_1$ not containing $q$. Let $R_2$ be the completion of $R_1$ at $\m$ and let $K$ be the field of fractions.  It  is enough to show that the characters $\phi_1,\ldots,\phi_n$ form a regular sequence  for $A_K=K\otimes_k A_{k,q}(n)$. Thus we are reduced to the situation in which $k$ is a $p$-adic number field  and $q\in R\backslash \m$, where $R$ is the ring of $p$-adic integers in $k$ and $\m$ is its maximal ideal.

\q Let $F=R/\m$ be the residue field and let $\barq=q+\m\in F$. Let $A_R$ be the $R$-subalgebra of $A_K$ generated by all the coefficient elements $c_{ij}$. We have the natural map $A_R\to \barA_R=A_R/\m A_R=
F\otimes_R A$.  For $a\in A_R$ we write $\bara$ for the image in $\barA$. Now $\barq$ is a root of unity, so by Step 2, $\barphi_1,\ldots,\barphi_n$ is an $\barA$-regular sequence.  Thus by Lemma 7.6, $\phi_1,\ldots,\phi_n$ is a regular sequence in $A_R$ and hence in $A=A_K$. 

\medskip

\it Step 5. \q\rm  There remains the case in which $q$ is not algebraic over the prime field. Thus, by Step 3,  we need only consider the case $k=k_0(t)$, the field of rational functions over a field $k_0$, with $q=t$. 
We write $R$ for the localisation of $k_0[t]$ at  $t-1$.  We write $A_R$ for the $R$-subalgebra of $A_{k,q}(n)$ generated by all coefficient elements  $c_{ij}$ and $\barA$ for $A_{k_0,1}(n)$. We have the natural map $A_R\to \barA$. For $a\in A$ we write $\bara$ for the image in $\barA$.  The elements $\barphi_1,\ldots,\barphi_n$ form a regular sequence for $\barA$, by Step 1. Hence,  by Lemma 7.6,  $\phi_1,\ldots,\phi_n$ is a regular sequence for $A_R$, and hence for $A=A_{k,t}(n)$. 
\end{proof}

\q We now fix a field $k$ and $0\neq q\in k$. We write $A$ for $A_{k,q}(n)$ and $J$ for $J_{k,q}(n)$. We now write $\barA$ for $A/J^+A$. Then $J^+A$ is a homogenous ideal and $\barA$ inherits a grading from $A$.

\q We work in the set $\Gamma$ of functions $f:\eno \to \zed$ such that $f(0)=1$. Now $\Gamma$ is an abelian  group with the convolution product 
$$(f*g)(r)=\sum_{r=u+v} f(u)g(v)$$
with unit $1_\Gamma$, whose value on $0$ is $1$ and whose value on all   positive integers is $0$.

\q We consider the functions $f, \barf\in \Gamma$, given by $f(r)= \dim A_r$ and $\barf(r)=\dim \barA_r$, $r\in \eno$.  

\q We set $[1,n]=\{1,\ldots,n\}$ and for $S\subseteq [1,n]$ define $d(S)=\sum_{s\in S} s$.  We have the Koszul resolution 
$$0\to F(n)\to \cdots \to F(0)\to \barA\to 0\eqno{(6)}$$
of $\barA$, defined by the $A$-regular sequence $\phi_1,\ldots,\phi_n$, as in \cite[Chapter 7]{Matsumura}.   For $0\leq p\leq n$,  the  $A$-module $F(p)$ is free on $e_S$,  as $S$ ranges over subsets of $[1,n]$ of size $p$.  Each $F(p)$ is graded with $A_v e_S\subseteq F(p)_r$, for $v+ d(S)=r$.  Hence, for $r\geq 0$,  we have
$$\dim F(p)_r= \sum_{S\subseteq [1,n], |S|=p, v+ d(S)=r} \dim A_v.$$
Moreover, the maps in $(6)$ preserve the grading so we have
$$\dim \barA_r =\sum_{p=0}^n (-1)^ p \sum_{S\subseteq [1,n], |S|=p, v+ d(S)=r} \dim A_v$$
i.e.,
$$\dim \barA_r= \sum_{S\subseteq [1,n],  v+ d(S)=r} (-1)^{|S|} \dim A_v$$
i.e.,
$$\dim \barA_r = \sum_{r=u+v}\, \sum_{S\subseteq [1,n],  d(S)=u} (-1)^{|S|} \dim A_v \eqno{(7)}$$

\bs\bs

Thus, defining $g\in \Gamma$ by $g(t)= \sum_{S\subseteq [1,n],  d(S)=t} (-1)^{|S|}$, for $t\in \eno$, we have
$$\barf = g*f.\eqno{(8)}$$

\q However, we also have the Koszul resolution
$$\cdots \to J \to J/J^+\to 0$$
of $k=J/J^+$.  Arguing as above we get the following analogue of (7): 
$$ \sum_{r=u+v} \sum_{S\subseteq [1,n],  d(S)=u} (-1)^{|S|} \dim J_v=
\begin{cases}1, &  \hbox{ if } r=0; \\
0, & \hbox{otherwise}
\end{cases}$$
Thus we have $g*h=1_\Gamma$, where $h(t)=\dim J_t$, $t\in \eno$.  Now taking the convolution product of (8) with $h$ we get $f=h* \barf$. Hence for $r\in \eno$ we have
$$\dim A_r =\sum_{r=u+v} \dim J_u \dim \barA_v. \eqno{(9)}$$

\begin{proposition} For  $r\geq 0$ write $A_r=H_r\oplus (AJ^+)_r$, for some subspace $H_r$ of $A_r$ and put $H=\bigoplus_{r\geq 0} H_r$. Then the multiplication map $J\otimes H\to A$ is a linear isomoprhism.

\q In particular (for all $0\neq q\in k$)  $A$ is a free $J$-module 
\end{proposition}

\begin{proof} We have $\dim H_r =\dim \barA_r$, for $r\geq 0$. By induction one obtains that $A_r=\sum_{r=s+t} J_s H_t$, for all $r\geq 0$. Hence multiplication $\bigoplus_{r=s+t} J_s\otimes H_t \to A_r$ is surjective and hence, by (9),  an isomorphism, for $r\geq 0$.  Thus the multiplication map $J\otimes H\to A$ is an isomorpism. 
\end{proof}

\begin{corollary} Suppose $q$ is not a root of unity. We regard $A(n)$ as a $G(n)$-module via the adjoint action. Then there exists a homogenous $G(n)$-submodule $H$ of $A(n)$ such that the multiplication map 
$J\otimes H\to A$ is an isomorphism.
\end{corollary}

\begin{proof}  In this case all $G(n)$-modules are completely reducible so we may choose $H_r$ as in the Proposition, to be a $G(n)$-submodule. Then $H=\bigoplus_{r=0}^\infty H_r$ has the required property.
\end{proof}

\q We get the following version of the main result, Theorem A, of \cite{Richardson}, by taking $G(n)$-module isotypic components. 

\begin{corollary} Suppose $q$ is not a root of unity. We regard $A(n)$ as a $G(n)$-module via the adjoint action.  Let $\lambda$ be a dominant weight and write $B(\lambda)$ for the sum of all irreducible $G(n)$-submodules of $A(n)$ isomorphic to $L(\lambda)$. Then,  there exists a  finite dimensional submodule $Z(\lambda)$ of $A(n)$ such that multiplication  $J\otimes Z(\lambda)\to B(\lambda)$ is an isomorphism. 

\end{corollary}

\begin{remark} One immediately get the analogous results for $k[G(n)]$ by localizing at the determinant $d$ (when $q$ is not a root of unity).  There is a $G(n)$-submodule $H$ of $k[G(n)]$ such that multiplication $C(kG(n)])\otimes H\to k[G(n)]$ is an isomorphism. Moreover,  if $\lambda$ is a dominant weight an $B(\lambda)$ is sum of all irreducible $G(n)$-submodules of $k[G(n)]$ isomorphic to $L(\lambda)$ then there exists a finite dimensional submodule $Z(\lambda)$ of $k[G(n)]$ such that multiplication  $C(k[G(n)]) \otimes Z(\lambda)\to B(\lambda)$ is an isomorphism.
\end{remark}

\bs\bs\bs\bs

%--------------------------------------------------

%biblio.tex

%-------------------------------

\end{document}